\documentclass[a4paper,12pt,reqno]{amsart}
\usepackage{amsfonts, color}
\usepackage{amsmath}
\usepackage{amssymb}
\usepackage[a4paper]{geometry}
\usepackage{mathrsfs}
\usepackage[colorlinks]{hyperref}
\DeclareMathOperator*{\esssup}{ess\,sup}
\usepackage{hyperref}
\renewcommand\eqref[1]{(\ref{#1})}

\makeatletter
\newcommand*{\mint}[1]{%
  \mint@l{#1}{}%
}
\newcommand*{\mint@l}[2]{%
  \@ifnextchar\limits{%
    \mint@l{#1}%
  }{%
    \@ifnextchar\nolimits{%
      \mint@l{#1}%
    }{%
      \@ifnextchar\displaylimits{%
        \mint@l{#1}%
      }{%
        \mint@s{#2}{#1}%
      }%
    }%
  }%
}
\newcommand*{\mint@s}[2]{%
  \@ifnextchar_{%
    \mint@sub{#1}{#2}%
  }{%
    \@ifnextchar^{%
      \mint@sup{#1}{#2}%
    }{%
      \mint@{#1}{#2}{}{}%
    }%
  }%
}
\def\mint@sub#1#2_#3{%
  \@ifnextchar^{%
    \mint@sub@sup{#1}{#2}{#3}%
  }{%
    \mint@{#1}{#2}{#3}{}%
  }%
}
\def\mint@sup#1#2^#3{%
  \@ifnextchar_{%
    \mint@sup@sub{#1}{#2}{#3}%
  }{%
    \mint@{#1}{#2}{}{#3}%
  }%
}
\def\mint@sub@sup#1#2#3^#4{%
  \mint@{#1}{#2}{#3}{#4}%
}
\def\mint@sup@sub#1#2#3_#4{%
  \mint@{#1}{#2}{#4}{#3}%
}
\newcommand*{\mint@}[4]{%
  \mathop{}%
  \mkern-\thinmuskip
  \mathchoice{%
    \mint@@{#1}{#2}{#3}{#4}%
        \displaystyle\textstyle\scriptstyle
  }{%
    \mint@@{#1}{#2}{#3}{#4}%
        \textstyle\scriptstyle\scriptstyle
  }{%
    \mint@@{#1}{#2}{#3}{#4}%
        \scriptstyle\scriptscriptstyle\scriptscriptstyle
  }{%
    \mint@@{#1}{#2}{#3}{#4}%
        \scriptscriptstyle\scriptscriptstyle\scriptscriptstyle
  }%
  \mkern-\thinmuskip
  \int#1%
  \ifx\\#3\\\else_{#3}\fi
  \ifx\\#4\\\else^{#4}\fi
}
\newcommand*{\mint@@}[7]{%
  \begingroup
    \sbox0{$#5\int\m@th$}%
    \sbox2{$#5\int_{}\m@th$}%
    \dimen2=\wd0 %
    \let\mint@limits=#1\relax
    \ifx\mint@limits\relax
      \sbox4{$#5\int_{\kern1sp}^{\kern1sp}\m@th$}%
      \ifdim\wd4>\wd2 %
        \let\mint@limits=\nolimits
      \else
        \let\mint@limits=\limits
      \fi
    \fi
    \ifx\mint@limits\displaylimits
      \ifx#5\displaystyle
        \let\mint@limits=\limits
      \fi
    \fi
    \ifx\mint@limits\limits
      \sbox0{$#7#3\m@th$}%
      \sbox2{$#7#4\m@th$}%
      \ifdim\wd0>\dimen2 %
        \dimen2=\wd0 %
      \fi
      \ifdim\wd2>\dimen2 %
        \dimen2=\wd2 %
      \fi
    \fi
    \rlap{%
      $#5%
        \vcenter{%
          \hbox to\dimen2{%
            \hss
            $#6{#2}\m@th$%
            \hss
          }%
        }%
      $%
    }%
  \endgroup
}
\numberwithin{equation}{section}
\theoremstyle{plain}
\newtheorem{thm}{Theorem}[section]

\newtheorem{cor}[thm]{Corollary}

\theoremstyle{definition}
\newtheorem{defn}[thm]{Definition}
\newtheorem{rem}[thm]{Remark}

\title[Wave equation for singular Sturm-Liouville operator]{Wave equation for Sturm-Liouville operator with singular intermediate coefficient and potential}
\author[M. Ruzhansky]{Michael Ruzhansky}
\address{
  Michael Ruzhansky:
  \endgraf
  Department of Mathematics: Analysis, Logic and Discrete Mathematics
  \endgraf
  Ghent University, Belgium
  \endgraf
 and
  \endgraf
  School of Mathematical Sciences
  \endgraf
  Queen Mary University of London
  \endgraf
  United Kingdom
  \endgraf
  {\it E-mail address} {\rm michael.ruzhansky@ugent.be}
  }

\author[A. Yeskermessuly]{Alibek Yeskermessuly}
\address{
  Alibek Yeskermessuly:
    \endgraf
  Department of Mathematics: Analysis, Logic and Discrete Mathematics
  \endgraf
  Ghent University, Belgium
  \endgraf
 and 
   \endgraf
  Altynsarin Arkalyk Pedagogical Institute 
  \endgraf
  Arkalyk, Kazakhstan
  \endgraf
  {\it E-mail address} {\rm alibek.yeskermessuly@gmail.com}}

\begin{document}

\thanks{The authors are supported by the FWO Odysseus 1 grant G.0H94.18N: Analysis and Partial Differential Equations and by the Methusalem programme of the Ghent University Special Research Fund (BOF) (Grant number 01M01021). Michael Ruzhansky is also supported by EPSRC grants EP/R003025/2 and EP/V005529/1, and the second author by the international internship program \textquotedblleft Bolashak\textquotedblright \,of the Republic of Kazakhstan. }
\maketitle              
\begin{abstract}
In this paper, we consider a wave equation on a bounded domain with a Sturm-Liouville operator with a singular intermediate coefficient and a singular potential. To obtain and evaluate the solution, the method of separation of variables is used, then the expansion in the Fourier series in terms of the eigenfunctions of the Sturm-Liouville operator is used. The Sturm-Liouville eigenfunctions are determined by such coefficients using the modified Prufer transform. Existence, uniqueness and consistency theorems are also proved for a very weak solution of the wave equation with singular coefficients.
\end{abstract}
\section{Introduction}

The purpose of this work is to establish the results on the well-posedness of the wave equation for the Sturm-Liouville operator with a singular intermediate coefficient and a singular potential.

In \cite{Lan-Ham}, very weak solutions of the wave equation for the Landau Hamiltonian with an irregular electromagnetic field are obtained for an unbounded domain in the space. A number of works (\cite{ARST1}, \cite{ARST2}, \cite{ARST3}, \cite{CRT1}, \cite{CRT2}, \cite{CRT3}, \cite{Garet}, \cite{R-Y}) are also devoted to this topic. The difference between our results is that we consider the problem in a bounded domain. We have obtained similar results in the work \cite{R-Sh-Y}, so our current work is a further development of these results, allowing one to include the intermediate term.

It is well known that the wave equation is easily reduced to ordinary linear equations by the ``separation of variables" method (see, for example, \cite{Separ}).

To obtain the main results, we present some information about the Sturm-Liouville operator with singular potentials. Savchuk and Shkalikov in \cite{Sav-Shk} obtained eigenvalues and eigenfunctions of the Sturm-Liouville operator with singular potentials. This method was further developed in the works \cite{N-zSk}, \cite{Savch}, \cite{Sav-Shk2}, \cite{SV}. We are guided by this method and will develop it with the addition of an intermediate coefficient, and, accordingly, additional conditions will be imposed on the coefficients, and the regularity requirements will be relaxed.

In particular, we consider the problem of constructing eigenvalues and eigenfunctions of the Sturm-Liouville operator $\mathcal{L}$ generated on the interval (0,1) by the differential expression
\begin{equation}\label{St-L}
    \mathcal{L}y:=-\frac{d^2}{dx^2}y+p(x)\frac{d}{dx}y+q(x)y
\end{equation}
with the boundary conditions
\begin{equation}\label{Dirihle}
    y(0)=y(1)=0. 
\end{equation}
We first assume that $p \in W^2_1(0,1)$ (summable squared with the first derivative), and that the potential $q$ is defined as
\begin{equation}\label{con-q}
    q(x)=\nu'(x), \qquad \nu\in L^2(0,1).
\end{equation}

We consider the eigenvalue equation $\mathcal{L}y=\lambda y$. Introducing the substitution
\begin{equation}\label{repl}
    y=\exp{\left\{\frac{1}{2}\int\limits_0^xp(\xi)d\xi\right\}}z,
\end{equation}
we get the equation
\begin{equation}\label{eq-z}
    -z''+q(x)z+\left(\frac{p^2(x)}{4}-\frac{p'(x)}{2}\right)z=\lambda z,
\end{equation}
while the boundary conditions do not change:
\begin{equation}\label{bou-z}
    z(0)=z(1)=0.
\end{equation}

We introduce the quasi-derivative in the following form
$$z^{[1]}(x)=z'(x)-\nu(x)z(x),$$
then equation \eqref{eq-z} transforms to the equation
\begin{equation}\label{quasi}
    -\left(z^{[1]}\right)'-\nu(x)z^{[1]}+\left(-\nu^2(x)+\frac{p^2(x)}{4}-\frac{p'(x)}{2}\right)z=\lambda z.
\end{equation}
We introduce
$$\mathbf{z}(x)=\left(\begin{array}{c}
    z(x)  \\
    z^{[1]}(x) 
\end{array}\right)=\left(\begin{array}{c}
    \psi_1(x)  \\
    \psi_2(x) 
\end{array}\right),\qquad A=\left(\begin{array}{ccc}
    \nu & &1 \\
    -\nu^2+\frac{p^2}{4}-\frac{p'}{2}-\lambda & & -\nu
\end{array}\right),$$
then we pass from the \eqref{quasi} to the system
$$\mathbf{z}'(x)=A\mathbf{z}.$$

We make the substitution
$$\psi_1(x)=r(x)\sin \theta(x),\qquad \psi_2(x)=\lambda^\frac{1}{2}r(x)\cos \theta(x),$$
which is a modification of the Prufer substitution (\cite{Ince}). Here we have
\begin{equation}\label{theta}
\theta'(x,\lambda)=\lambda^\frac{1}{2}+\nu(x)\sin 2\theta(x,\lambda) +\lambda^{-\frac{1}{2}}\left(\nu^2(x)-\frac{p^2(x)}{4}+\frac{p'(x)}{2}\right)\sin^2 \theta(x,\lambda),
\end{equation}
\begin{equation}\label{r}
    r'(x,\lambda)=-r(x,\lambda)\left[\nu(x)\cos 2\theta(x,\lambda)+\frac{\lambda^{-\frac{1}{2}}}{2}\left(\nu^2(x)-\frac{p^2(x)}{4}+\frac{p'(x)}{2}\right)\sin 2\theta(x,\lambda)\right].
\end{equation}

The solution of the equation \eqref{theta} will be sought in the form $\theta(x,\lambda)=\lambda^\frac{1}{2} x+\eta(x,\lambda), $ where
$$\eta(x,\lambda)=\int\limits_0^x\nu(s)\sin 2\theta(s,\lambda)ds+\lambda^{-\frac{1}{2}}\int\limits_0^x\left(\nu^2(s)-\frac{p^2(s)}{4}+\frac{p'(s)}{2}\right)\sin^2 \theta(s,\lambda)ds.$$
Using the method of successive approximations, it is easy to show that this equation has a solution that is uniformly bounded for $0\leq x\leq 1$ and $\lambda\geq 1$. Since $p\in W^1_2(0,1), \, \nu\in L^2(0,1)$ and $\nu^2\in L^1(0,1)$, by virtue of the Riemann-Lebesgue lemma $\eta(x,\lambda)=o(1)$ at $\lambda \to \infty$. Therefore,
$$\theta(x,\lambda)=\lambda^\frac{1}{2}x+o(1),$$
moreover $\theta(0,\lambda)=0.$

Using the Riemann-Lebesgue lemma again, from equation \eqref{r} we find
\begin{eqnarray*}
r(x,\lambda)=\exp{\left(-\int\limits_0^x\nu(s)\cos 2\theta(s,\lambda)ds-\frac{\lambda^{-\frac{1}{2}}}{2}\int\limits_0^x\left(\nu^2(s)-\frac{p^2(s)}{4}+\frac{p'(s)}{2}\right)\sin 2\theta(s,\lambda)ds\right)}.
\end{eqnarray*}

Using the boundary conditions \eqref{bou-z} we obtain
$$ \psi_1(1,\lambda)=r(1,\lambda)\sin\theta(1,\lambda)=0,\,\, r(1,\lambda)\neq 0,\,\,\theta(1,\lambda)=\pi n.$$
Then the eigenvalues of the equation \eqref{eq-z} with the boundary conditions \eqref{bou-z} are given by
\begin{equation}\label{e-val}
    \lambda_n=(\pi n)^2(1+o(n^{-1})),\qquad n=1,2,...,
\end{equation}
and the corresponding eigenfunctions are
\begin{equation}\label{sol-SL}
    \Tilde{\psi}_n(x)=r_n(x)\sin(\sqrt{\lambda_n}x +\eta_n(x)).
\end{equation}

The first derivatives of $\Tilde{\psi_n}$ are then given by the formulas
\begin{equation}\label{phi-der}
    \Tilde{\psi}'_n(x)=\sqrt{\lambda_n}r_n(x)\cos(\theta_n(x))+\nu(x)\Tilde{\psi}_n(x).
\end{equation}

Let us estimate $\|\Tilde{\psi}_n\|_{L^2}$ using the formula \eqref{sol-SL} as follows
\begin{eqnarray}\label{est-high}
\|\Tilde{\psi}_n\|^2_{L^2}&=&\int\limits_0^1\left|r_n(x)\sin\left(\lambda_n^{\frac{1}{2}}x+\eta_n(x)\right)\right|^2dx\leq \int\limits_0^1\left|r_n(x)\right|^2dx\nonumber\\
&\leq& \int\limits_0^1\left|\exp\left(-\int\limits_0^x \nu(s)\cos{2\theta_n(s)}ds\right.\right.\nonumber\\
&-&\left.\left.\frac{1}{2}\frac{1}{\sqrt{\lambda_n}}\int\limits_0^x\left(\nu^2(s)-\frac{p^2(s)}{4}+\frac{p'(s)}{2}\right)\sin{2\theta_n(s)}ds\right)\right|^2dx\nonumber\\
&\lesssim& \int\limits_0^1\exp{\left(2\int\limits_0^x|\nu(s)|ds+\frac{1}{\sqrt{\lambda_n}}\left(\int\limits_0^x|\nu^2(s)|ds+\int\limits_0^x|p^2(s)|ds+\int\limits_0^x|p'(s)|ds\right)\right)}dx\nonumber\\
&\lesssim& \exp{\left(\|\nu\|_{L^1}+\lambda^{-\frac{1}{2}}_n\left(\|\nu\|^2_{L^2}+\|p\|^2_{L^2}+\|p'\|_{L^1}\right)\right)}<\infty,
\end{eqnarray}
since $\nu\in L^2(0,1)$, $p\in W^2_1(0,1)$ and $\lambda_n\to \infty$ at $n\to \infty$.

Also, according to Theorem 4 in \cite{Sav-Shk}, we have 
\begin{equation}\label{est_low}
  \Tilde{\psi}_n(x)=\sin(\pi nx)+o(1)  
\end{equation}
for sufficiently large $n$, it means that there exist some $C_0>0$, such that $C_0<\|\Tilde{\phi}_n\|_{L^2}<\infty$. 
Since the eigenfunctions \eqref{sol-SL} form an orthogonal basis in $L^2(0,1)$, we normalize them for further use
\begin{equation}\label{norm}
    \psi_n(x)=\frac{\Tilde{\psi}_n(x)}{\sqrt{\langle\Tilde{\psi}_n,\Tilde{\psi}_n\rangle}}=\frac{\Tilde{\psi}_n(x)}{\|\Tilde{\psi}_n\|_{L^2}}.
\end{equation}

Returning again to the substitution \eqref{repl}, we obtain the eigenfunctions
\begin{equation}\label{eig-f}
    \phi_n(x)=\exp{\left\{\frac{1}{2}\int\limits_0^xp(\xi)d\xi\right\}}\psi_n(x)
\end{equation}
of the operator $\mathcal{L}$ generated by the differential expression \eqref{St-L} with the boundary conditions \eqref{Dirihle}. In this case, the eigenvalues remain as \eqref{e-val}. It should be noted that the eigenfunctions $\phi_n$ are orthogonal in the weighted space $L^2_g(0,1)$ with norm
$$\|\phi_n\|_{L^2_g}^2=\int\limits_0^1\left|g(x)\phi_n(x)\right|^2dx,$$
where
$$g(x)=\exp{\left\{-\frac{1}{2}\int\limits_0^xp(\xi)d\xi\right\}}.$$

Let us estimate the norm of $\phi_n$ in $L^2(0,1)$ as
\begin{eqnarray}\label{norm-phi}
\|\phi_n\|^2_{L^2}&=&\int\limits_0^1\exp{\left\{\int\limits_0^xp(\xi)d\xi\right\}}|\psi_n(x)|^2dx\nonumber\\
&\leq& \exp{\left\{\int\limits_0^1|p(x)|dx\right\}}\int\limits_0^1|\psi_n(x)|^2dx\leq\exp{\left\{\|p\|_{L^1}\right\}}<\infty,    
\end{eqnarray}
since $p\in W^2_1(0,1)$ and $\|\psi_n\|_{L^2}=1$.

\section{Main results}

We consider the wave equation
\begin{equation}\label{C.p1}
         \partial^2_t u(t,x)+\mathcal{L} u(t,x)=0,\qquad (t,x)\in [0,T]\times (0,1),
\end{equation}
with initial conditions
\begin{equation}\label{C.p2} \left\{\begin{array}{l}u(0,x)=u_0(x),\,\,\, x\in (0,1), \\
\partial_t u(0,x)=u_1(x), \,\,\, x\in (0,1),\end{array}\right.\end{equation}
and with Dirichlet boundary conditions
\begin{equation}\label{C.p3}
u(t,0)=0=u(t,1),\qquad t\in [0,T],
\end{equation}
where  $\mathcal{L}$ is defined by

\begin{equation}\label{1}
    \mathcal{L} u(t,x):=-\partial^2_x u(t,x)+p(x)\partial_xu(t,x)+ q(x)u(t,x),\qquad x\in(0,1),
\end{equation}
where $p\in W^2_1(0,1)$, and $q$ is defined as in \eqref{con-q}. 

In our results below, concerning the initial/boundary problem \eqref{C.p1}-\eqref{C.p3}, as the preliminary step we first carry out the analysis in the strictly regular case for summable $q \in L^2(0,1)$. In this case, we obtain the well-posedness in the Sobolev spaces $W^k_\mathcal{L}$ associated to the operator $\mathcal{L}$: we define the Sobolev spaces $W^k_\mathcal{L}$ associated to $\mathcal{L}$, for any $k \in \mathbb{R}$, as the space  
$$W^k_{\mathcal{L}}:=\left\{f\in \mathcal{D}'_\mathcal{L}(0,1):\,\mathcal{L}^{k/2}f\in L^2(0,1)\right\},$$
with the norm $\|f\|_{W^k_{\mathcal{L}}}:=\|\mathcal{L}^{k/2}f\|_{L^2}$. The global space of distributions $\mathcal{D}'_\mathcal{L}(0,1)$ is defined as bellow.

The space $C^\infty_\mathcal{L}(0,1):=\mathrm{Dom}(\mathcal{L}^\infty)$ is called the space of test functions for $\mathcal{L}$, where we define 
$$\mathrm{Dom}(\mathcal{L}^\infty):=\bigcap\limits_{m=1}^\infty \mathrm{Dom}(\mathcal{L}^m),$$
where $\mathrm{Dom}(\mathcal{L}^m)$ is the domain of the operator $\mathcal{L}^m$, in turn defined as
$$\mathrm{Dom}(\mathcal{L}^m):=\left\{f\in L^2(0,1): \mathcal{L}^j f\in \mathrm{Dom}(\mathcal{L}),\,\, j=0,1,2,...,m-1\right\}.$$
The Fréchet topology of $C^\infty_\mathcal{L}(0,1)$ is given by the family of norms 
\begin{equation}\label{frechet}
    \|\phi\|_{C^m_\mathcal{L}}:=\max\limits_{j\leq m}\|\mathcal{L}^j\phi\|_{L^2(0,1)},\quad m\in \mathbb{N}_0,\,\, \phi\in C^\infty_\mathcal{L}(0,1).
\end{equation}
The space of $\mathcal{L}$-distributions
$$\mathcal{D}'_\mathcal{L}:=\mathbf{L}\left(C^\infty_\mathcal{L}(0,1),\mathbb{C}\right)$$
is the space of all linear continuous functionals on $C^\infty_\mathcal{L}(0,1)$. For $\omega \in \mathcal{D}'_\mathcal{L}(0,1)$ and $\phi\in C^\infty_\mathcal{L}(0,1)$, we shall write 
$$\omega(\phi)=\langle \omega, \phi\rangle.$$
For any $\psi \in C^\infty_\mathcal{L}(0,1)$, the functional 
$$C^\infty_\mathcal{L}(0,1)\ni \phi \mapsto \int\limits_0^1 \psi(x)\phi(x)dx$$
is an $\mathcal{L}$-distribution, which gives an embedding $\psi \in C^\infty_\mathcal{L}(0,1)\hookrightarrow \mathcal{D}'_\mathcal{L}(0,1)$.

We introduce the spaces $C^j([0,T],W^k_\mathcal{L}(0,1))$, given by the family of norms
\begin{equation}
    \|f\|_{C^n([0,T],W^k_\mathcal{L}(0,1))}=\max\limits_{0\leq t\leq T}\sum\limits_{j=0}^n\left\|\partial^j_t f(t,\cdot)\right\|_{W^k_\mathcal{L}},
\end{equation}
where $k\in \mathbb{R}, \, f\in C^n([0,T],W^k_\mathcal{L}(0,1)).$
 
\begin{thm}\label{th1}
Assume that $p'\in L^2(0,1)$, $q=\nu'$, $\nu \in L^\infty(0,1)$. For any $k\in \mathbb{R}$ if the initial data satisfy $(u_0,\, u_1) \in W^{1+k}_{\mathcal{L}}\times W^k_{\mathcal{L}}$ then the wave equation \eqref{C.p1} with the initial/boundary problem  \eqref{C.p2}-\eqref{C.p3}  has unique solution $u\in C([0,T], W^{1+k}_{\mathcal{L}})\cap C^1([0,T], W^{k}_{\mathcal{L}})$. It satisfies the estimates
\begin{equation}\label{est1}
    \|u(t,\cdot)\|^2_{L^2}\lesssim \exp{\left\{\|p\|_{L^1}\right\}}\left(\|gu_0\|^2_{L^2}+\|gu_1\|^2_{W^{-1}_{\mathcal{L}}}\right),
\end{equation}
\begin{equation}\label{est2}
\|\partial_t u(t,\cdot)\|^2_{L^2}\lesssim \exp{\left\{\|p\|_{L^1}\right\}}\left(\|gu_0\|^2_{W^1_{\mathcal{L}}}+\|gu_1\|^2_{L^2}\right),
\end{equation}
\begin{eqnarray}\label{est3}
\|\partial_x u(t,\cdot)\|^2_{L^2}
&\lesssim& \exp{\left\{\|p\|_{L^1}\right\}}\left\{\left(1+\|\nu\|^2_{L^2}\left(\|\nu\|^2_{L^2}+\|p\|^2_{L^2}+\|p'\|^2_{L^1}\right)\right)\left(\|gu_0\|^2_{W^1_{\mathcal{L}}}\right.\right.\nonumber\\
&+& \left.\|gu_1\|^2_{L^2}\Big)+\left(\|p\|^2_{L^\infty}+\|\nu\|^2_{L^\infty}\right)\left(\|gu_0\|^2_{L^2}+\|gu_1\|^2_{W^{-1}_{\mathcal{L}}}\right)\right\},
\end{eqnarray}
\begin{eqnarray}\label{est4}
\left\|\partial^2_xu(t,\cdot)\right\|^2_{L^2} 
&\lesssim &\exp{\{\|p\|_{L^1}\}}\left\{\|p\|^2_{L^\infty}\left(\left(1+\|\nu\|^2_{L^2}\left(\|\nu\|^2_{L^2}+\|p\|^2_{L^2}+\|p'\|^2_{L^1}\right)\right)\times\right.\right.\nonumber\\
&\times& \left.\left(\|gu_0\|^2_{W^1_{\mathcal{L}}}+\|gu_1\|^2_{L^2}\right)+\left(\|p\|^2_{L^\infty}+\|\nu\|^2_{L^\infty}\right)\left(\|gu_0\|^2_{L^2}+\|gu_1\|^2_{W^{-1}_{\mathcal{L}}}\right)\right)\nonumber\\
&+&\left.\|q\|^2_{L^\infty}\left(\|gu_0\|^2_{L^2}+\|gu_1\|^2_{W^{-1}_{\mathcal{L}}}\right)+\left\|gu_0\right\|^2_{W^2_{\mathcal{L}}}+\|gu_1\|^2_{W^1_{\mathcal{L}}}\right\},
\end{eqnarray}
\begin{equation}\label{est5}
    \|u(t,\cdot)\|^2_{W^k_\mathcal{L}} \lesssim \exp{\left\{\|p\|_{L^1}\right\}}\left(\left\|gu_0\right\|^2_{W^k_{\mathcal{L}}}+\left\|gu_1\right\|^2_{W^{k-1}_{\mathcal{L}}}\right),
\end{equation}
where the constants in these inequalities are independent of $u_0$, $u_1$, $p$ and $q$.
\end{thm} 

We note that $p'\in L^2(0,1)$ implies that $\|p\|_{L^\infty}\leq |p(0)|+\|p'\|_{L^2(0,1)}.$
Indeed, if $p'\in L^2(0,1)$, then 
$$|p(x)|=\left|\int\limits_0^xp'(\xi)d\xi+p(0)\right|\leq |p(0)|+\|p'\|_{L^2}<\infty.$$

\begin{proof}
Let us apply the technique of the separation of variables (see, e.g. \cite{Separ}). This method involves finding a solution of a certain form. In particular, we are looking for a solution of the form
$$u(t,x)=T(t)X(x),$$
for functions $T(t)$, $X(x)$ to be determined. Suppose we can find a solution of \eqref{C.p1} of this form. Plugging a function $u(t,x)=T(t)X(x)$ into the wave equation, we arrive at the equation
$$T''(t)X(x)-T(t)X''(x)+p(x)T(x)X'(x)+q(x)T(t)X(x)=0,$$
Dividing this equation by $T(t)X(x)$, we have
\begin{equation}\label{3-1}
    \frac{T''(t)}{T(t)}=\frac{X''(x)-p(x)X'(x)-q(x)X(x)}{X(x)}=-\lambda,
\end{equation}
for some constant $\lambda$. Therefore, if there exists a solution $u(t,x) = T(t)X(x)$ of the wave equation, then $T(t)$ and $X(x)$ must satisfy the equations
$$\frac{T''(t)}{T(t)}=-\lambda,$$
$$\frac{X''(x)-p(x)X'(x)-q(x)X(x)}{X(x)}=-\lambda,$$
for some constant $\lambda$. In addition, in order for $u$ to satisfy the boundary conditions \eqref{C.p3}, we need our function $X$ to satisfy the boundary conditions \eqref{Dirihle}. That is, we need to find a function $X$ and a scalar $\lambda$, such that
\begin{equation}\label{4}
    -X''(x)+p(x)X'(x)+q(x)X(x)=\lambda X(x),
\end{equation}
\begin{equation}\label{5}
    X(0)=X(1)=0.
\end{equation}

The equation \eqref{4} with the boundary conditions \eqref{5} has the eigenvalues of the form \eqref{e-val} with the corresponding eigenfunctions of the form \eqref{eig-f} of the Sturm-Liouville operator $\mathcal{L}$ generated by the differential expression \eqref{St-L}.

Further, we solve the left hand side of the equation \eqref{3-1} respect to the independent variable $t$,
\begin{equation}\label{3}
        T''(t)=-\lambda T(t),  \qquad t\in [0,T].
\end{equation}

It is well known (\cite{Separ}) that the solution of the equation \eqref{3} with the initial conditions \eqref{C.p2} is
$$T(t)=A_n \cos \sqrt{\lambda_n}t+\frac{1}{\sqrt{\lambda_n}}B_n \sin\left(\sqrt{\lambda_n}t\right).$$
Then the solution of equation \eqref{C.p1} is given by
\begin{equation}\label{part-sol}
    u(t,x)=\left(A_n\cos\left(\sqrt{\lambda_n}t\right)+\frac{1}{\sqrt{\lambda_n}}B_n \sin\left(\sqrt{\lambda_n}t\right)\right)\phi_n(x).
\end{equation}
For each value of $n$ equation \eqref{part-sol} is a solution. By the superposition principle the sum of all these solution is also a solution
\begin{equation}\label{23}
    u(t,x)=\sum\limits_{n=1}^\infty\left(A_n\cos\left(\sqrt{\lambda_n}t\right)+\frac{1}{\sqrt{\lambda_n}}B_n \sin\left(\sqrt{\lambda_n}t\right)\right)\phi_n(x).
\end{equation}
Applying the initial conditions to equation \eqref{23}, we have
\begin{equation}\label{u0u1}
u_0(x)=\sum\limits_{n=1}^\infty A_n\phi_n(x),\qquad u_1(x)=\sum\limits_{n=1}^\infty B_n\phi_n(x),
\end{equation}
multiplying both sides of each equation in \eqref{u0u1} by $g(x)\psi_m(x)$, we get
\begin{equation}\label{u0u1-g}
\begin{array}{l}
     u_0(x)g(x)\psi_m(x)=\sum\limits_{n=1}^\infty A_n\psi_n(x)\psi_m(x),\\
     u_1(x)g(x)\psi_m(x)=\sum\limits_{n=1}^\infty B_n\psi_n(x)\psi_m(x).
\end{array}
\end{equation}
Note that
$$g(x)=\exp{\left\{-\frac{1}{2}\int\limits_0^xp(\xi)d\xi\right\}},\quad \phi_n(x)g(x)=\psi_n(x).$$
Integrating over $(0,1)$ in \eqref{u0u1-g}, taking into account the orthonormality of $\psi_n$ in $L^2(0,1)$, we obtain
$$A_n=\int\limits_0^1u_0(x)g(x)\psi_n(x)dx, \quad B_n=\int\limits_0^1 u_1(x)g(x)\psi_n(x)dx.$$

Further we will prove that $u\in C^2([0,T],L^2(0,1))$. By using the Cauchy-Schwarz inequality and fixed $t$, we can deduce that

\begin{eqnarray}\label{25}
\|u(t, \cdot)\|^2_{L^2}&=&\int\limits_0^1|u(t,x)|^2dx \nonumber\\
&=&\int\limits_0^1\left|\sum\limits_{n=1}^\infty\left[A_n \cos \sqrt{\lambda_n} t+\frac{1}{\sqrt{\lambda_n}} B_n\sin\sqrt{\lambda_n} t\right]\phi_n(x)\right|^2dx\nonumber\\
&\lesssim& \int\limits_0^1\sum\limits_{n=1}^\infty\left|A_n \cos\sqrt{\lambda_n} t+\frac{1}{\sqrt{\lambda_n} }B_n\sin\sqrt{\lambda_n} t\right|^2|\phi_n(x)|^2dx\nonumber\\
&\leq& \int\limits_0^1\sum\limits_{n=1}^\infty\left(|A_n||\phi_n(x)|+\frac{1}{\sqrt{\lambda_n}}|B_n||\phi_n(x)|\right)^2 dx\nonumber\\
&\lesssim& \sum\limits_{n=1}^\infty\left(\int\limits_0^1|A_n|^2|\phi_n(x)|^2dx+\int\limits_0^1\left|\frac{B_n}{\sqrt{\lambda_n}}\right|^2|\phi_n(x)|^2dx\right). 
\end{eqnarray}
By using the Parseval identity and taking into account \eqref{norm-phi},
we get
\begin{eqnarray}\label{An1}
\sum\limits_{n=1}^\infty \int\limits_0^1|A_n|^2|\phi_n(x)|^2dx&\leq& \exp{\left\{\|p\|_{L^1}\right\}}\sum\limits_{n=1}^\infty |A_n|^2\nonumber\\
&=&\exp{\left\{\|p\|_{L^1}\right\}}\sum\limits_{n=1}^\infty\left|\int\limits_0^1u_0(x)g(x)\psi_n(x)dx\right|^2\nonumber\\
&=&\exp{\left\{\|p\|_{L^1}\right\}}\sum\limits_{n=1}^\infty\left|\langle (g u_0), \psi_n\rangle\right|^2\leq \exp{\left\{\|p\|_{L^1}\right\}}\|g u_0\|^2_{L^2}\nonumber\\
&\leq&\exp{\left\{\|p\|_{L^1}\right\}}\|g\|^2_{L^\infty}\|u_0\|^2_{L^2}.
\end{eqnarray}
Now, let us estimate $\|g\|_{L^\infty}$, where
$$g(x)=\exp{\left\{-\frac{1}{2}\int\limits_0^xp(\xi)d\xi\right\}}.$$
If $p\geq0$ at $x\in (0,1)$, then $\|g\|_{L^\infty}=1$. Otherwise, when we do not have $p\geq0$, then
\begin{equation}\label{g-exp}
\|g\|^2_{L^\infty}=\esssup_{x\in(0, 1)}|g(x)|^2\leq \exp{\left\{\int\limits_0^1|p(x)|dx\right\}}=\exp{\left\{\|p\|_{L^1}\right\}}.
\end{equation}
According to the last expressions, we have
\begin{eqnarray}\label{An}
\sum\limits_{n=1}^\infty \int\limits_0^1|A_n|^2|\phi_n(x)|^2dx
&\leq&\exp{\left\{\|p\|_{L^1}\right\}}\|g\|^2_{L^\infty}\|u_0\|^2_{L^2}\nonumber\\
&\leq&\exp{\left\{2\|p\|_{L^1}\right\}}\|u_0\|^2_{L^2}.
\end{eqnarray}

For the second term in \eqref{25}, using \eqref{norm-phi}, the properties of the eigenvalues of the operator $\mathcal{L}$ and the Parseval's identity, we obtain the following estimate 
\begin{eqnarray*}
\sum\limits_{n=1}^\infty \int\limits_0^1\left|\frac{B_n}{\sqrt{\lambda_n}}\right|^2|\phi_n(x)|^2dx&\leq& \exp{\left\{\|p\|_{L^1}\right\}}\sum\limits_{n=1}^\infty \left|\frac{B_n}{\sqrt{\lambda_n}}\right|^2\\
&=&\exp{\left\{\|p\|_{L^1}\right\}}\sum\limits_{n=1}^\infty \left|\int\limits_0^1\frac{1}{\sqrt{\lambda_n}}u_1(x)g(x)\psi_n(x)dx\right|^2\\
&=&\exp{\left\{\|p\|_{L^1}\right\}}\sum\limits_{n=1}^\infty \left|\langle gu_1,\mathcal{L}^{-\frac{1}{2}}\psi_n\rangle\right|^2\\
&=&\exp{\left\{\|p\|_{L^1}\right\}}\sum\limits_{n=1}^\infty \left|\langle \mathcal{L}^{-\frac{1}{2}}\left(gu_1\right),\psi_n\rangle\right|^2\\
&=&\exp{\left\{\|p\|_{L^1}\right\}}\left\|\mathcal{L}^{-\frac{1}{2}}\left(gu_{1}\right)\right\|^2_{L^2}\leq \exp{\left\{\|p\|_{L^1}\right\}}\left\|gu_{1}\right\|^2_{W^{-1}_{\mathcal{L}}}.
\end{eqnarray*}

Therefore
$$
\|u(t,\cdot)\|^2_{L^2}\lesssim \exp{\left\{\|p\|_{L^1}\right\}}\left(\|gu_0\|^2_{L^2}+\|gu_1\|^2_{W^{-1}_{\mathcal{L}}}\right).
$$

Now, let us estimate
\begin{eqnarray}\label{t26}
\|\partial_t u(t,\cdot)\|^2&=&\int\limits_0^1|\partial_tu(t,x)|^2dt\nonumber\\
&=&\int\limits_0^1\left|\sum\limits_{n=1}^\infty\left[-\sqrt{\lambda_n}A_n\sin\left(\sqrt{\lambda_n}t\right)+\frac{1}{\sqrt{\lambda_n}}\sqrt{\lambda_n}B_n\cos \sqrt{\lambda_n} t\right]\phi_n(x)\right|^2dx \nonumber\\ 
&\lesssim& \exp{\left\{\|p\|_{L^1}\right\}}\left(\sum\limits_{n=1}^\infty|\sqrt{\lambda_n} A_n |^2+\sum\limits_{n=1}^\infty|B_n|^2\right).
\end{eqnarray}
The second term of \eqref{t26} gives the norm of $\|gu_1\|^2_{L^2}$ by the Parseval identity. Since $\lambda_n$ are eigenvalues and $\phi_n$ are eigenfunctions of the operator $\mathcal{L}$, we obtain 
\begin{eqnarray}\label{21-1}\sum\limits_{n=1}^\infty|\sqrt{\lambda_n}A_n|^2&=& \sum\limits_{n=1}^\infty\left|\sqrt{\lambda_n}\int\limits_0^1 g(x)u_0(x)\psi_n(x)dx\right|^2 \nonumber\\
&\leq& \sum\limits_{n=1}^\infty\left| \int\limits_0^1 \mathcal{L}^\frac{1}{2}\left(gu_0\right)\psi_n(x)dx\right|^2.
      \end{eqnarray}
It is known by Parseval's identity that
$$\sum\limits_{n=1}^\infty\left| \int\limits_0^1 \mathcal{L}^\frac{1}{2}\left(gu_0\right)\psi_n(x)dx\right|^2=\|\mathcal{L}^\frac{1}{2}\left(gu_0\right)\|^2_{L^2}=\|gu_0\|^2_{W^1_{\mathcal{L}}}.$$
Thus, 
$$\|\partial_t u(t,\cdot)\|^2_{L^2}\lesssim \exp{\left\{\|p\|_{L^1}\right\}}\left(\|gu_0\|^2_{W^1_{\mathcal{L}}}+\|gu_1\|^2_{L^2}\right).$$

We now consider the next estimate for the derivative
\begin{eqnarray*}
\|\partial_x u(t,\cdot)\|^2_{L^2}&=&\int\limits_0^1|\partial_xu(t,x)|^2dt\\
&=&\int\limits_0^1\left|\sum\limits_{n=1}^\infty\left[A_n\cos\left(\sqrt{\lambda_n}t\right)+\frac{1}{\sqrt{\lambda_n}}B_n\sin\left(\sqrt{\lambda_n}t\right)\right]\phi'_n(x)\right|^2dx,
\end{eqnarray*}
where $\phi'(x)$, taking into account \eqref{sol-SL}, \eqref{norm} and \eqref{eig-f}, is given by
\begin{eqnarray}\label{phi'}
    \phi'_n(x)&=&\left(\exp{\left\{\frac{1}{2}\int\limits_0^xp(\xi)d\xi\right\}}\psi_n(x)\right)'=\exp{\left\{\frac{1}{2}\int\limits_0^xp(\xi)d\xi\right\}}\times\nonumber\\
    &\times&\left(\frac{\sqrt{\lambda_n}r_n(x)}{\|\Tilde{\psi}_n\|_{L^2}}\cos{\theta_n(x)}+\left(\frac{p(x)}{2}+\nu(x)\right)\psi_n(x)\right).
\end{eqnarray}
By using formula \eqref{phi'} let us estimate
\begin{eqnarray*}
\|\partial_x u(t,\cdot)\|^2_{L^2}&=&\int\limits_0^1\left|\sum\limits_{n=1}^\infty\left[A_n\cos \sqrt{\lambda_n}t+\frac{1}{\sqrt{\lambda_n}}B_n\sin\left(\sqrt{\lambda_n}t\right)\right]\exp{\left\{\frac{1}{2}\int\limits_0^xp(\xi)d\xi\right\}}\times \right. \nonumber\\
&\times&\left.\left(\frac{\sqrt{\lambda_n}r_n(x)}{\|\Tilde{\psi}_n\|_{L^2}}\cos{\theta_n(x)}+\left(\frac{p(x)}{2}+\nu(x)\right)\psi_n(x)\right)\right|^2dx\nonumber\\
&\lesssim&\sum\limits_{n=1}^\infty\left[|A_n|^2+\left|\frac{1}{\sqrt{\lambda_n}}B_n\right|^2\right]\int\limits_0^1\exp{\left\{\int\limits_0^xp(\xi)d\xi\right\}}\left|\frac{\sqrt{\lambda_n}r_n(x)}{\|\Tilde{\psi}_n\|_{L^2}}\right|^2dx  \nonumber\\
&+&\sum\limits_{n=1}^\infty\left[|A_n|^2+\left|\frac{1}{\sqrt{\lambda_n}}B_n\right|^2\right]\times\\
&\times&\int\limits_0^1\exp{\left\{\int\limits_0^xp(\xi)d\xi\right\}}\left|\left(\frac{p(x)}{2}+\nu(x)\right)\psi_n(x)\right|^2dx.
\end{eqnarray*}
According \eqref{est-high} and \eqref{est_low}, there exist some $C_0>0$, such that $C_0<\|\Tilde{\psi}_n\|_{L^2}<\infty$, and taking into account \eqref{norm-phi} we get
\begin{eqnarray}\label{38}
\|\partial_x u(t,\cdot)\|^2_{L^2}&\lesssim&\exp{\left\{\|p\|_{L^1}\right\}}\sum\limits_{n=1}^\infty\left[\sqrt{\lambda_n}A_n|^2+\left|B_n\right|^2\right]\int\limits_0^1\left|r_n(x)\right|^2dx+ \exp{\left\{\|p\|_{L^1}\right\}}\times \nonumber\\
&\times&\sum\limits_{n=1}^\infty\left[|A_n|^2+\left|\frac{1}{\sqrt{\lambda_n}}B_n\right|^2\right]\int\limits_0^1\left|\left(\frac{p(x)}{2}+\nu(x)\right)\psi_n(x)\right|^2dx.
\end{eqnarray}
We follow the proof of Lemma 1 in \cite{Savch} to obtain
$$r_n(x)=1+\rho_n(x),\quad \|\rho_n\|^2_{L^2}\lesssim \left(1+\|\nu\|^2_{L^2}\right)\left(\|\nu\|^2_{L^2}+\|p\|^2_{L^2}+\|p'\|^2_{L^1}\right),$$
where the constant is independent of $\nu$ and $n$. Then
\begin{equation}\label{39}
  \|r_n\|^2_{L^2}\lesssim 1+\|\nu\|^2_{L^2}\left(\|\nu\|^2_{L^2}+\|p\|^2_{L^2}+\|p'\|^2_{L^1}\right).  
\end{equation}
For the second term we obtain
\begin{equation}\label{40}
  \int\limits_0^1\left|\left(\frac{p(x)}{2}+\nu(x)\right)\psi_n(x)\right|^2dx\lesssim \left(\|p\|^2_{L^\infty}+\|\nu\|^2_{L^\infty}\right)\|\psi_n\|^2_{L^2}=\|p\|^2_{L^\infty}+\|\nu\|^2_{L^\infty} , 
\end{equation}
since $\{\psi_n\}$ is an orthonormal basis in $L^2$. Using the last relations we can obtain the estimate for $\|\psi'_n\|_{L^2}$ as the following form
\begin{eqnarray}\label{psi'-norm}
\|\psi'_n\|^2&\lesssim& \exp{\left\{\|p\|_{L^1}\right\}}\left(\sqrt{\lambda_n}\left(1+\|\nu\|^2_{L^2}\left(\|\nu\|^2_{L^2}+\|p\|^2_{L^2}+\|p'\|^2_{L^1}\right)\right)\right.\nonumber\\
&+&\|p\|^2_{L^\infty}+\|\nu\|^2_{L^\infty}\Big).
\end{eqnarray}
Using \eqref{38}, \eqref{39}, \eqref{40}, \eqref{21-1} and \eqref{An1} we obtain
\begin{eqnarray}\label{u_x}
\|\partial_x u(t,\cdot)\|^2_{L^2}&\lesssim&
\exp{\left\{\|p\|_{L^1}\right\}}\left[\left(1+\|\nu\|^2_{L^2}\left(\|\nu\|^2_{L^2}+\|p\|^2_{L^2}+\|p'\|^2_{L^1}\right)\right)\left(\|gu_0\|^2_{W^1_{\mathcal{L}}}\right.\right.\nonumber\\
&+& \left.\|gu_1\|^2_{L^2}\Big)+\left(\|p\|^2_{L^\infty}+\|\nu\|^2_{L^\infty}\right)\left(\|gu_0\|^2_{L^2}+\|gu_1\|^2_{W^{-1}_{\mathcal{L}}}\right)\right].
\end{eqnarray}

Let us get next estimates by using that $\phi''_n(x)=p(x)\phi'_n(x)+(q(x)-\lambda_n)\phi_n(x)$, since $\phi_n$ is a normalised eigenfunction for $\mathcal{L}$ with eigenvalue $\lambda_n$. We have
\begin{eqnarray}\label{u_x2}
\left\|\partial_x^2u(t, \cdot)\right\|^2_{L^2}&=&\int\limits_0^1\left|\partial^2_xu(t,x)\right|^2dx\nonumber\\
&=&\int\limits_0^1\left|\sum\limits_{n=1}^\infty \left[A_n \cos \sqrt{\lambda_n}t+\frac{1}{\sqrt{\lambda_n}} B_n \sin\left(\sqrt{\lambda_n}t\right)\right]\phi''_n(x)\right|^2dx\nonumber\\
&\lesssim&\int\limits_0^1\sum\limits_{n=1}^\infty \left[\left|A_n\right|^2 +\left|\frac{B_n}{\sqrt{\lambda_n}}\right|^2 \right]\left|p(x)\phi'_n(x)+(q(x)-\lambda_n)\phi_n(x)\right|^2dx\nonumber\\
&\lesssim&\int\limits_0^1\sum\limits_{n=1}^\infty \left[\left|A_n\right|^2 + \left|\frac{B_n}{\sqrt{\lambda_n}}\right|^2 \right]\left|p(x)\phi'_n(x)\right|^2dx\nonumber\\
&+&\int\limits_0^1\sum\limits_{n=1}^\infty \left[\left|A_n\right|^2 +\left|\frac{B_n}{\sqrt{\lambda_n}}\right|^2 \right]|(q(x)-\lambda_n)\phi_n(x)|^2dx=J_1+J_2.
\end{eqnarray}
By using \eqref{phi'}-\eqref{u_x} we get
\begin{eqnarray}\label{u_xx1}
J_1&:=&\int\limits_0^1\sum\limits_{n=1}^\infty \left[\left|A_n\right|^2 +\frac{1}{\lambda_n} \left|B_n\right|^2 \right]\left|p(x)\phi'_n(x)\right|^2dx\nonumber\\
&=&\int\limits_0^1|p(x)|^2\sum\limits_{n=1}^\infty \left[\left|A_n\right|^2 +\frac{1}{\lambda_n} \left|B_n\right|^2 \right]\left|\phi'_n(x)\right|^2dx\nonumber\\
&\lesssim&\exp{\{\|p\|_{L^1}\}}\|p\|^2_{L^\infty}\left[\left(1+\|\nu\|^2_{L^2}\left(\|\nu\|^2_{L^2}+\|p\|^2_{L^2}+\|p'\|^2_{L^1}\right)\right)\left(\|gu_0\|^2_{W^1_{\mathcal{L}}}\right.\right.\nonumber\\
&+& \left.\|gu_1\|^2_{L^2}\Big)+\left(\|p\|^2_{L^\infty}+\|\nu\|^2_{L^\infty}\right)\left(\|gu_0\|^2_{L^2}+\|gu_1\|^2_{W^{-1}_{\mathcal{L}}}\right)\right].
\end{eqnarray}
Let us estimate the second term of \eqref{u_x2},
\begin{eqnarray}\label{u_xx}
J_2&:=&\int\limits_0^1\sum\limits_{n=1}^\infty \left[\left|A_n\right|^2 + \left|\frac{B_n}{\sqrt{\lambda_n}}\right|^2 \right]|(q(x)-\lambda_n)\phi_n(x)|^2dx\nonumber\\
&\lesssim&\exp{\{\|p\|_{L^1}\}}\sum\limits_{n=1}^\infty \left(|A_n|^2 +\left|\frac{B_n}{\sqrt{\lambda_n}}\right|^2\right)\int\limits_0^1|q(x)\psi_n(x)|^2dx+\nonumber\\
&+&\exp{\{\|p\|_{L^1}\}}\sum\limits_{n=1}^\infty \left(|\lambda_n A_n|^2 +\left|\sqrt{\lambda_n} B_n\right|^2\right)\int\limits_0^1|\psi_n(x)|^2dx\nonumber\\
&\leq&\exp{\{\|p\|_{L^1}\}}\left(\|q\|^2_{L^\infty}\sum\limits_{n=1}^\infty \left(|A_n|^2 +\left|\frac{B_n}{\sqrt{\lambda_n}}\right|^2\right)\right.\nonumber\\
&+&\left.\sum\limits_{n=1}^\infty \left|\lambda_n A_n\right|^2 +\sum\limits_{n=1}^\infty\left|\sqrt{\lambda_n}B_n\right|^2\right). 
\end{eqnarray}
Using the property of the operator $\mathcal{L}$ and the Parseval identity for the last expression in \eqref{u_xx}, we obtain
\begin{eqnarray*}
 \sum\limits_{n=1}^\infty\left|\sqrt{\lambda_n}B_n\right|^2&=& \sum\limits_{n=1}^\infty\left|\int\limits_0^1\sqrt{\lambda_n}g(x)u_1(x)\psi_n(x)dx\right|^2
 \leq\sum\limits_{n=1}^\infty\left|\int\limits_0^1\mathcal{L}^\frac{1}{2}\left(gu_1\right)\psi_n(x)dx\right|^2\\
 &=&\left\|\mathcal{L}^\frac{1}{2}\left(gu_1\right)\right\|^2_{L^2}=\|gu_1\|^2_{W^1_{\mathcal{L}}}.
\end{eqnarray*}
Taking into account the last expression and \eqref{u_xx1}, \eqref{u_xx} we obtain
\begin{eqnarray*}
\left\|\partial^2_xu(t,\cdot)\right\|^2_{L^2} 
&\lesssim &\exp{\{\|p\|_{L^1}\}}\|p\|^2_{L^\infty}\Big[\left(1+\|\nu\|^2_{L^2}\left(\|\nu\|^2_{L^2}+\|p\|^2_{L^2}+\|p'\|^2_{L^1}\right)\right)\times\nonumber\\
&\times& \left.\left(\|gu_0\|^2_{W^1_{\mathcal{L}}}+\|gu_1\|^2_{L^2}\right)+\left(\|p\|^2_{L^\infty}+\|\nu\|^2_{L^\infty}\right)\left(\|gu_0\|^2_{L^2}+\|gu_1\|^2_{W^{-1}_{\mathcal{L}}}\right)\right]\\
&+&\exp{\{\|p\|_{L^1}\}}\left(\|q\|^2_{L^\infty} \left(\|gu_0\|^2_{L^2}+\|gu_1\|^2_{W^{-1}_{\mathcal{L}}}\right)+\left\|gu_0\right\|^2_{W^2_{\mathcal{L}}}+\|gu_1\|^2_{W^1_{\mathcal{L}}}\right).
\end{eqnarray*}

Let us carry out the last estimate \eqref{est5} using that $\mathcal{L}^ku=\lambda_n^ku$ and Parseval's identity,
\begin{eqnarray*}
 \left\|u(t, \cdot)\right\|^2_{W^k_\mathcal{L}}&=&\left\|\mathcal{L}^\frac{k}{2}u(t, \cdot)\right\|^2_{L^2}=\int\limits_0^1\left|\mathcal{L}^\frac{k}{2}u(t,x)\right|^2dx=\int\limits_0^1\left|\lambda_n^\frac{k}{2}u(t,x)\right|^2dx\\
&=&\int\limits_0^1\left|\sum\limits_{n=1}^\infty \left[A_n \cos \sqrt{\lambda_n}t+\frac{1}{\sqrt{\lambda_n}} B_n \sin\left(\sqrt{\lambda_n}t\right)\right]\lambda_n^\frac{k}{2}\phi_n(x)\right|^2dx\\
&\lesssim&\exp{\left\{\|p\|_{L^1}\right\}}\sum\limits_{n=1}^\infty\left( \left| \lambda_n^\frac{k}{2}A_n\right|^2+ \left| \lambda_n^\frac{k-1}{2}B_n\right|^2\right)\\
&\leq&\exp{\left\{\|p\|_{L^1}\right\}}\left(\left\|\mathcal{L}^\frac{k}{2}\left(gu_0\right)\right\|^2_{L^2}+\left\|\mathcal{L}^\frac{k-1}{2}\left(gu_1\right)\right\|^2_{L^2}\right)\\
&=&\exp{\left\{\|p\|_{L^1}\right\}}\left(\left\|gu_0\right\|^2_{W^k_{\mathcal{L}}}+\left\|gu_1\right\|^2_{W^{k-1}_{\mathcal{L}}}\right).
\end{eqnarray*}
The proof of Theorem \ref{th1} is complete.
\end{proof}

We will now express all the estimates in terms of the coefficients, to be used in the very weak well-posedness in Section \ref{ch4}.

\begin{cor}\label{cor1}
Assume that $p'\in L^2(0,1)$, $q=\nu'$, $\nu \in L^\infty(0,1)$. If the initial data satisfy $(u_0,\, u_1) \in L^2(0,1)\times L^2(0,1)$ and $(u_0'', \, u''_1)\in L^2(0,1)\times L^2(0,1)$, then the wave equation \eqref{C.p1} with the initial/boundary problems  \eqref{C.p2}-\eqref{C.p3}  has unique solution $u\in C([0,T], L^2(0,1))$ which satisfies the estimates
\begin{equation}\label{ec1}
    \|u(t,\cdot)\|^2_{L^2}\lesssim \exp{\{2\|p\|_{L^1}\}}\left(\|u_0\|^2_{L^2}+\|u_1\|^2_{L^2}\right),
\end{equation}
\begin{eqnarray}\label{ec2}
\|\partial_t u(t,\cdot)\|^2_{L^2}&\lesssim& \exp{\{2\|p\|_{L^1}\}}\left( \|u''_0\|^2_{L^2}+\|p\|^2_{L^\infty}\|u'_0\|^2_{L^2}\right.\nonumber\\
&+&\left.\left(\|p\|^4_{L^\infty}+\|p'\|^2_{L^\infty}+\|q\|^2_{L^\infty}\right)\|u_0\|^2_{L^2}+\|u_1\|^2_{L^2}\right),
\end{eqnarray}
\begin{eqnarray}\label{ec3}
\|\partial_x u(t,\cdot)\|^2_{L^2}
&\lesssim& \exp{\left\{2\|p\|_{L^1}\right\}}\left\{\left(1+\|\nu\|^2_{L^2}\left(\|\nu\|^2_{L^2}+\|p\|^2_{L^2}+\|p'\|^2_{L^1}\right)\right)\times\right.\nonumber\\
&\times& \left(\|u''_0\|^2_{L^2}+\|p\|^2_{L^\infty}\|u'_0\|^2_{L^2}+\left(\|p\|^4_{L^\infty}+\|p'\|^2_{L^\infty}+\|q\|^2_{L^\infty}\right)\|u_0\|^2_{L^2}\right.\nonumber\\
&+&\left.\|u_1\|^2_{L^2}\right)+\left.\left(\|p\|^2_{L^\infty}+\|\nu\|^2_{L^\infty}\right)\left(\|u_0\|^2_{L^2}+\|u_1\|^2_{L^2}\right)\right\},
\end{eqnarray}
\begin{eqnarray}\label{ec4}
\left\|\partial_x^2u(t, \cdot)\right\|^2_{L^2}&\lesssim& \exp{\left\{2\|p\|_{L^1}\right\}}\left\{\left(1+\|\nu\|^2_{L^2}\left(\|\nu\|^2_{L^2}+\|p\|^2_{L^2}+\|p'\|^2_{L^1}\right)\right)\times\right.\nonumber\\
&\times& \left(\|u''_0\|^2_{L^2}+\|p\|^2_{L^\infty}\|u'_0\|^2_{L^2}+\left(\|p\|^4_{L^\infty}+\|p'\|^2_{L^\infty}+\|q\|^2_{L^\infty}\right)\|u_0\|^2_{L^2}\right.\nonumber\\
&+&\left.\|u_1\|^2_{L^2}\right)+\left(\|p\|^2_{L^\infty}+\|\nu\|^2_{L^\infty}\right)\left(\|u_0\|^2_{L^2}+\|u_1\|^2_{L^2}\right)\nonumber\\
&+&\|u''_0\|^2_{L^2}+\|u''_1\|^2_{L^2}+\|p\|^2_{L^\infty}\left(\|u'_0\|^2_{L^2}+\|u'_1\|^2_{L^2}\right)\nonumber\\
&+&\left.\left(\|p\|^4_{L^\infty}+\|p'\|^2_{L^\infty}+\|q\|^2_{L^\infty}\right)\left(\|u_0\|^2_{L^2}+\|u_1\|^2_{L^2}\right)\right\},
\end{eqnarray}
where the constants in these inequalities are independent of $u_0$, $u_1$, $p$ and $q$.
\end{cor}
\begin{proof}
By using inequality \eqref{25} we obtain
\begin{eqnarray}\label{26}
\|u(t, \cdot)\|^2_{L^2}&\lesssim& \sum\limits_{n=1}^\infty\left(\int\limits_0^1|A_n|^2|\phi_n(x)|^2dx+\int\limits_0^1\left|\frac{B_n}{\sqrt{\lambda_n}}\right|^2|\phi_n(x)|^2dx\right).
\end{eqnarray}
In Theorem \ref{th1} we obtained estimates with respect to the operator $\mathcal{L}$, but here we want to obtain estimates with respect to the initial data $(u_0,\, u_1)$ and functions $p$ and $q$. Therefore, since $\lambda_n\geq 1$ we can use the next estimate
\begin{equation}\label{B-lam}
    \int\limits_0^1\left|\frac{B_n}{\sqrt{\lambda_n}}\right|^2|\phi_n(x)|^2dx\leq \int\limits_0^1|B_n|^2|\phi_n(x)|^2dx.
\end{equation}
Thus, using \eqref{An} and the Parseval identity  in \eqref{26}, taking into account the last relation, we obtain
$$\|u(t, \cdot)\|^2_{L^2}\lesssim \exp{\{\|p\|_{L^1}\}}\left(\sum\limits_{n=1}^\infty\left(|A_n|^2+|B_n|^2\right) \right)\leq  \exp{\{2\|p\|_{L^1}\}}\left(\|u_0\|^2_{L^2}+\|u_1\|^2_{L^2}\right).$$
By \eqref{t26} we have
\begin{eqnarray*}
\|\partial_t u(t,\cdot)\|^2 &\lesssim& \exp{\{\|p\|_{L^1}\}}\left(\sum\limits_{n=1}^\infty|\sqrt{\lambda_n} A_n |^2+\sum\limits_{n=1}^\infty|B_n|^2\right).
\end{eqnarray*}
Since $\lambda_n$ are eigenvalues of the operator $\mathcal{L}$, we obtain 
\begin{eqnarray*}\label{21}\sum\limits_{n=1}^\infty|\sqrt{\lambda_n}A_n|^2&\lesssim&\sum\limits_{n=1}^\infty\left| \int\limits_0^1 \lambda_n gu_0(x)\psi_n(x)dx\right|^2\nonumber\\
&=&\sum\limits_{n=1}^\infty\left| \int\limits_0^1 \left(-(gu_0)''(x)+p(x)(gu_0)'(x)+q(x)(gu_0)(x)\right)\psi_n(x)dx\right|^2\\
&\lesssim& \sum\limits_{n=1}^\infty\left| \int\limits_0^1 (gu_0)''(x)\psi_n(x)dx\right|^2+\sum\limits_{n=1}^\infty\left| \int\limits_0^1p(x)(gu_0)'(x)\psi_n(x)dx\right|^2\\
&+&\sum\limits_{n=1}^\infty\left|\int\limits_0^1q(x)(gu_0)(x)\psi_n(x)dx\right|^2.
      \end{eqnarray*}
Since $p,\,q\in L^\infty(0,1)$ and by Parseval's identity, we get
\begin{eqnarray}\label{LA}
\sum\limits_{n=1}^\infty|\sqrt{\lambda_n}A_n|^2&\lesssim&\sum\limits_{n=1}^\infty |\langle (gu_0)'',\psi_n\rangle|^2+\sum\limits_{n=1}^\infty |\langle p(gu_0)',\psi_n\rangle|^2+\sum\limits_{n=1}^\infty |\langle q(gu_0),\psi_n\rangle|^2\nonumber\\
&=&\|(gu_0)''\|^2_{L^2}+\|p(gu_0)'\|^2_{L^2}+\|q(gu_0)\|^2_{L^2}\nonumber\\
&\leq&\|(gu_0)''\|^2_{L^2}+\|p\|^2_{L^\infty}\|(gu_0)'\|^2_{L^2}+\|q\|^2_{L^\infty}\|q(gu_0)\|^2_{L^2},
\end{eqnarray}
thus, 
$$\|\partial_t u(t,\cdot)\|^2_{L^2}\lesssim \exp{\{\|p\|_{L^1}\}}\left( \|(gu_0)''\|^2_{L^2}+\|p\|^2_{L^\infty}\|(gu_0)'\|^2_{L^2}+\|q\|^2_{L^\infty}\|gu_0\|^2_{L^2}+\|gu_1\|^2_{L^2}\right).$$

To obtain the results of Section \ref{ch4}, we need estimates in terms of $p$, $q$, and $(u_0,\,u_1)$. Therefore, we proceed to the next estimates. We have
\begin{eqnarray*}
\|(gu_0)'\|^2&\lesssim& \|g'u_0\|^2_{L^2}+\|gu'_0\|^2_{L^2}\leq \|g'\|^2_{L^\infty}\|u_0\|^2_{L^2}+\|g\|^2_{L^\infty}\|u'_0\|^2_{L^2},
\end{eqnarray*}
where
$$g'(x)=-\frac{1}{2}p(x)\exp{\left\{-\frac{1}{2}\int\limits_0^xp(\xi)d\xi\right\}}=-\frac{1}{2}p(x)g(x),$$
and according to \eqref{g-exp} we obtain
\begin{eqnarray}\label{gu'}
\|(gu_0)'\|^2&\lesssim&\|pg\|^2_{L^\infty}\|u_0\|^2_{L^2}+\|g\|^2_{L^\infty}\|u'_0\|^2_{L^2}\leq \|g\|^2_{L^\infty}\left(\|p\|^2_{L^\infty}\|u_0\|^2_{L^2}+\|u'_0\|^2_{L^2}\right)\nonumber\\
&\leq&\exp{\left\{\|p\|_{L^1}\right\}}\left(\|p\|^2_{L^\infty}\|u_0\|^2_{L^2}+\|u'_0\|^2_{L^2}\right).
\end{eqnarray}
For $(gu_0)''$ one can obtain
\begin{eqnarray}\label{gu''}
\|(gu_0)''\|^2_{L^2}&\lesssim& \|g''u_0\|^2_{L^2}+\|g'u'_0\|^2_{L^2}+\|gu''_0\|^2_{L^2}\leq \|(p^2+p')g\|^2_{L^\infty}\|u_0\|^2_{L^2}\nonumber\\
&+&\|pg\|^2_{L^\infty}\|u_0'\|^2_{L^2}+\|g\|^2_{L^\infty}\|u''_0\|^2_{L^2}\lesssim\exp{\left\{\|p\|_{L^1}\right\}}\times\nonumber\\
&\times&\left(\left(\|p\|^4_{L^\infty}+\|p'\|^2_{L^\infty}\right)\|u_0\|^2_{L^2}+\|p\|^2_{L^\infty}\|u'_0\|^2_{L^2}+\|u''_0\|^2_{L^2}\right).
\end{eqnarray}
Given estimates \eqref{gu'}, \eqref{gu''} and \eqref{g-exp}, for $\|\partial_tu(t,\cdot)\|_{L^2}$ we get
\begin{eqnarray*}
\|\partial_t u(t,\cdot)\|^2_{L^2}&\lesssim&\exp{\{2\|p\|_{L^1}\}}\left( \|u''_0\|^2_{L^2}+\|p\|^2_{L^\infty}\|u'_0\|^2_{L^2}\right.\\
&+&\left.\left(\|p\|^4_{L^\infty}+\|p'\|^2_{L^\infty}+\|q\|^2_{L^\infty}\right)\|u_0\|^2_{L^2}+\|u_1\|^2_{L^2}\right).
\end{eqnarray*}

Taking \eqref{phi'}, \eqref{u_x}, \eqref{B-lam} and \eqref{LA} into account, we make the following estimates
\begin{eqnarray*}
\|\partial_x u(t,\cdot)\|^2_{L^2}&=&\int\limits_0^1|\partial_xu(t,x)|^2dt\\
&=&
\int\limits_0^1\left|\sum\limits_{n=1}^\infty\left[A_n\cos\left(\sqrt{\lambda_n}t\right)+\frac{1}{\sqrt{\lambda_n}}B_n\sin\left(\sqrt{\lambda_n}t\right)\right]\phi'_n(x)\right|^2dx\\
&\lesssim&\exp{\{\|p\|_{L^1}\}}\left(\sum\limits_{n=1}^\infty\left|\sqrt{\lambda_n}A_n \right|^2+\sum\limits_{n=1}^\infty|B_n|^2\right)\int\limits_0^1|r_n(x)|^2dx\\
&+&\exp{\{\|p\|_{L^1}\}}\left(\sum\limits_{n=1}^\infty|A_n|^2+\sum\limits_{n=1}^\infty\left|\frac{1}{\sqrt{\lambda_n}}B_n\right|^2\right)\times\\
&\times&\int\limits_0^1\left(|p(x)|^2+|\nu(x)|^2\right)|\psi_n(x)|^2dx\\
&\lesssim& \exp{\left\{\|p\|_{L^1}\right\}}\left\{\left(1+\|\nu\|^2_{L^2}\left(\|\nu\|^2_{L^2}+\|p\|^2_{L^2}+\|p'\|^2_{L^1}\right)\right)\times\right.\nonumber\\
&\times& \left(\|(gu_0)''\|^2_{L^2}+\|p\|^2_{L^\infty}\|(gu_0)'\|^2_{L^2}+\|q\|^2_{L^\infty}\|gu_0\|^2_{L^\infty}+\|gu_1\|^2_{L^2}\right)\\
&+&\left.\left(\|p\|^2_{L^\infty}+\|\nu\|^2_{L^\infty}\right)\left(\|gu_0\|^2_{L^2}+\|gu_1\|^2_{L^2}\right)\right\}.
\end{eqnarray*}
According to \eqref{gu'}, \eqref{gu''} and \eqref{g-exp} we get
\begin{eqnarray*}
\|\partial_x u(t,\cdot)\|^2_{L^2}&\lesssim&\exp{\left\{2\|p\|_{L^1}\right\}}\left\{\left(1+\|\nu\|^2_{L^2}\left(\|\nu\|^2_{L^2}+\|p\|^2_{L^2}+\|p'\|^2_{L^1}\right)\right)\times\right.\nonumber\\
&\times& \left(\|u''_0\|^2_{L^2}+\|p\|^2_{L^\infty}\|u'_0\|^2_{L^2}+\left(\|p\|^4_{L^\infty}+\|p'\|^2_{L^\infty}+\|q\|^2_{L^\infty}\right)\|u_0\|^2_{L^2}\right.\\
&+&\left.\|u_1\|^2_{L^2}\right)+\left.\left(\|p\|^2_{L^\infty}+\|\nu\|^2_{L^\infty}\right)\left(\|u_0\|^2_{L^2}+\|u_1\|^2_{L^2}\right)\right\}.
\end{eqnarray*}

Let us get an estimate for
\begin{eqnarray*}
\left\|\partial_x^2u(t, \cdot)\right\|^2_{L^2}&=&\int\limits_0^1\left|\partial^2_xu(t,x)\right|^2dx\\
&=&\int\limits_0^1\left|\sum\limits_{n=1}^\infty \left[A_n \cos \sqrt{\lambda_n}t+\frac{1}{\sqrt{\lambda_n}} B_n \sin\left(\sqrt{\lambda_n}t\right)\right]\phi''_n(x)\right|^2dx\\
&\lesssim& \int\limits_0^1\sum\limits_{n=1}^\infty \left(|A_n|^2|\phi''_n(x)|^2 +\left|\frac{B_n}{\sqrt{\lambda_n}}\right|^2|\phi''_n(x)|^2\right)dx\\
&\leq&\int\limits_0^1\sum\limits_{n=1}^\infty |A_n|^2|p(x)\phi'_n(x)+(q(x)-\lambda_n)\phi_n(x)|^2dx\\
&+&\int\limits_0^1\sum\limits_{n=1}^\infty \left|\frac{B_n}{\sqrt{\lambda_n}}\right|^2|p(x)\phi'_n(x)+(q(x)-\lambda_n)\phi_n(x)|^2dx=M_1+M_2. 
\end{eqnarray*}
We have
\begin{eqnarray*}
M_1&:=&\int\limits_0^1\sum\limits_{n=1}^\infty |A_n|^2|p(x)\phi'_n(x)+(q(x)-\lambda_n)\phi_n(x)|^2dx\\
&\lesssim&\int\limits_0^1|p(x)|^2\left(\sum\limits_{n=1}^\infty|A_n|^2|\phi'_n(x)|^2\right)dx+\int\limits_0^1|q(x)|^2\left(\sum\limits_{n=1}^\infty|A_n|^2|\phi_n(x)|^2\right)dx\\
 &+&\int\limits_0^1\sum\limits_{n=1}^\infty|\lambda_nA_n\phi_n(x)|^2dx,
\end{eqnarray*}
carrying out estimates as in \eqref{u_xx1} and \eqref{LA}, we obtain
\begin{eqnarray*}
M_1&\lesssim& \exp{\{\|p\|_{L^1}\}}\left(\|p\|^2_{L^\infty}\left(\left(1+\|\nu\|^2_{L^2}\left(\|\nu\|^2_{L^2}+\|p\|^2_{L^2}+\|p'\|^2_{L^1}\right)\right)\left( \|(gu_0)''\|^2_{L^2}\right.\right.\right.\\
&+&\left.\left.\|p\|^2_{L^\infty}\|(gu_0)'\|^2_{L^2}+\|q\|^2_{L^\infty}\|gu_0\|^2_{L^2}\right)+\left(\|p\|^2_{L^\infty}+\|\nu\|^2_{L^\infty}\right)\|gu_0\|^2_{L^2}\right)\\
&+&\left.\|q\|^2_{L^\infty}\|gu_0\|^2_{L^2}+\|(gu_0)''\|^2_{L^2}+\|p\|^2_{L^\infty}\|(gu_0)'\|^2_{L^2}+\|q\|^2_{L^\infty}\|gu_0\|^2_{L^2}\right).
\end{eqnarray*}
Similarly, we obtain the following estimate
\begin{eqnarray*}
M_2&:=&\int\limits_0^1\sum\limits_{n=1}^\infty \left|\frac{B_n}{\sqrt{\lambda_n}}\right|^2|p(x)\phi'_n(x)+(q(x)-\lambda_n)\phi_n(x)|^2dx\\
&\lesssim& \exp{\{\|p\|_{L^1}\}}\left(\|p\|^2_{L^\infty}\left(
\left(1+\|\nu\|^2_{L^2}\left(\|\nu\|^2_{L^2}+\|p\|^2_{L^2}+\|p'\|^2_{L^1}\right)\right)\|gu_1\|^2_{L^2}\right.\right.\\
&+&\left.\left(\|p\|^2_{L^\infty}+\|\nu\|^2_{L^\infty}\right)\|gu_1\|^2_{L^2}\right)+\|q\|^2_{L^\infty}\|gu_1\|^2_{L^2}+\|(gu_1)''\|^2_{L^2}\\
&+&\|p\|^2_{L^\infty}\|(gu_1)'\|^2_{L^2}+\left.\|q\|^2_{L^\infty}\|gu_1\|^2_{L^2}\right).
\end{eqnarray*}
Using  \eqref{gu'}, \eqref{gu''} and \eqref{g-exp}, we have
\begin{eqnarray*}
\left\|\partial_x^2u(t, \cdot)\right\|^2_{L^2}&\lesssim& 
\exp{\left\{2\|p\|_{L^1}\right\}}\left\{\left(1+\|\nu\|^2_{L^2}\left(\|\nu\|^2_{L^2}+\|p\|^2_{L^2}+\|p'\|^2_{L^1}\right)\right)\times\right.\nonumber\\
&\times& \left(\|u''_0\|^2_{L^2}+\|p\|^2_{L^\infty}\|u'_0\|^2_{L^2}+\left(\|p\|^4_{L^\infty}+\|p'\|^2_{L^\infty}+\|q\|^2_{L^\infty}\right)\|u_0\|^2_{L^2}\right.\\
&+&\left.\|u_1\|^2_{L^2}\right)+\left(\|p\|^2_{L^\infty}+\|\nu\|^2_{L^\infty}\right)\left(\|u_0\|^2_{L^2}+\|u_1\|^2_{L^2}\right)\\
&+&\|u''_0\|^2_{L^2}+\|u''_1\|^2_{L^2}+\|p\|^2_{L^\infty}\left(\|u'_0\|^2_{L^2}+\|u'_1\|^2_{L^2}\right)\\
&+&\left.\left(\|p\|^4_{L^\infty}+\|p'\|^2_{L^\infty}+\|q\|^2_{L^\infty}\right)\left(\|u_0\|^2_{L^2}+\|u_1\|^2_{L^2}\right)\right\}.
\end{eqnarray*}

The proof of Corollary \ref{cor1} is complete.
\end{proof}

\section{Non-homogeneous equation case}

In this section, we are going to give brief ideas for how to deal with the non-homogeneous wave equation with initial/boundary conditions
\begin{equation}\label{nonh}
    \left\{\begin{array}{l}
    \partial^2_t u(t,x)+\mathcal{L} u(t,x)=f(t,x),\qquad (t,x)\in [0,T]\times (0,1),\\
    u(0,x)=u_0(x),\quad x\in (0,1),\\
    \partial_tu(0,x)=u_1(x),\quad x\in(0,1),\\
    u(t,0)=0=u(t,1),\quad t\in[0,T],
    \end{array}\right.
\end{equation}
where operator $\mathcal{L}$ is defined by 
$$\mathcal{L}=-\frac{\partial^2}{\partial x^2}+p(x)\frac{\partial}{\partial x}+q(x),\qquad x\in(0,1).$$

\begin{thm}\label{non-hom}
Assume that $p'\in L^2(0,1)$, $q=\nu'$, $\nu \in L^\infty(0,1)$ and $f=f(t,x)\in C^1([0,T],L^2(0,1))$. For any $k\in \mathbb{R}$ if the initial data satisfy $(u_0,\, u_1) \in W^{1+k}_\mathcal{L}\times W^k_\mathcal{L}$ then the non-homogeneous wave equation with initial/boundary conditions \eqref{nonh} has unique solution $u\in C([0,T], W^{1+k}_\mathcal{L})\cap C^1([0,T], W^{k}_\mathcal{L})$ which satisfies the estimates
\begin{equation}\label{es-nh1}
        \|u(t,\cdot)\|^2_{L^2}\lesssim \exp{\left\{\|p\|_{L^1}\right\}}\left(\|gu_0\|^2_{L^2}+\|gu_1\|^2_{W^{-1}_{\mathcal{L}}}+2T^2\|g\|^2_{L^\infty}\|f\|^2_{C([0,T],L^2(0,1))}\right),
\end{equation}
\begin{equation}\label{es-nh2}
\|\partial_tu(t,\cdot)\|^2_{L^2}\lesssim \exp{\left\{\|p\|_{L^1}\right\}}\left(\|gu_0\|^2_{W^1_{\mathcal{L}}}+\|gu_1\|^2_{L^2}+2T^2\|g\|^2_{L^\infty}\|f\|^2_{C([0,T],L^2(0,1))}\right),
\end{equation}
\begin{eqnarray}\label{es-nh3}
\|\partial_xu(t,\cdot)\|^2_{L^2}&\lesssim& \exp{\left\{\|p\|_{L^1}\right\}}\Big\{\left(1+\|\nu\|^2_{L^2}\left(\|\nu\|^2_{L^2}+\|p\|^2_{L^2}+\|p'\|^2_{L^1}\right)\right)\times\nonumber\\
&\times& \left(\|gu_0\|^2_{W^1_{\mathcal{L}}}+\|gu_1\|^2_{L^2}\right)+\left(\|p\|^2_{L^\infty}+\|\nu\|^2_{L^\infty}\right)\left(\|gu_0\|^2_{L^2}+\|gu_1\|^2_{W^{-1}_{\mathcal{L}}}\right)\nonumber\\
&+&\left.\left(1+\|\nu\|^2_{L^2}\left(\|\nu\|^2_{L^2}+\|p\|^2_{L^2}+\|p'\|^2_{L^1}\right)+\|p\|^2_{L^\infty}+\|\nu\|^2_{L^\infty}\right)\times\right.\nonumber\\
&\times&\left.2T^2\|g\|^2_{L^\infty}\|f\|^2_{C([0,T],L^2(0,1))}\right\},
\end{eqnarray}
\begin{eqnarray}\label{es-nh4}
\|\partial^2_xu(t,\cdot)\|^2_{L^2}&\lesssim& 
\exp{\{\|p\|_{L^2}\}}\|p\|^2_{L^\infty}\Big\{\left(1+\|\nu\|^2_{L^2}\left(\|\nu\|^2_{L^2}+\|p\|^2_{L^2}+\|p'\|^2_{L^1}\right)\right)\times\nonumber\\
&\times& \left.\left(\|gu_0\|^2_{W^1_{\mathcal{L}}}+\|gu_1\|^2_{L^2}\right)+\left(\|p\|^2_{L^\infty}+\|\nu\|^2_{L^\infty}\right)\left(\|gu_0\|^2_{L^2}+\|gu_1\|^2_{W^{-1}_{\mathcal{L}}}\right)\right.\nonumber\\
&+&\|q\|^2_{L^\infty} \left(\|gu_0\|^2_{L^2}+\|gu_1\|^2_{W^{-1}_{\mathcal{L}}}\right)+\left\|gu_0\right\|^2_{W^2_{\mathcal{L}}}+\|gu_1\|^2_{W^1_{\mathcal{L}}}\nonumber\\
&+&\left(\|p\|^2_{L^\infty}\left(1+\|\nu\|^2_{L^2}\left(\|\nu\|^2_{L^2}+\|p\|^2_{L^2}+\|p'\|^2_{L^1}\right)+\|p\|^2_{L^\infty}+\|\nu\|^2_{L^\infty}\right)\right.\times\nonumber\\
&+&\left.\left.\|q\|^2_{L^\infty}\right)\|g\|^2_{L^\infty}\left(2T^2\|f\|^2_{C([0,T],L^2(0,1))}+T^2\|f\|^2_{C^1([0,T],L^2(0,1))}\right)\right\},
\end{eqnarray}
\begin{eqnarray}\label{es-nh5}
    \|u(t,\cdot)\|^2_{W^k_\mathcal{L}} &\lesssim& \exp{\left\{\|p\|_{L^1}\right\}}\left(\left\|gu_0\right\|^2_{W^k_{\mathcal{L}}}+\left\|gu_1\right\|^2_{W^{k-1}_{\mathcal{L}}}\right.\nonumber\\
    &+&\left.2T^2\left\|gf(\cdot,\cdot)\right\|^2_{C([0,T],W^{k-1}_\mathcal{L}(0,1))}\right),
\end{eqnarray}
where the constants in these inequalities are independent of $u_0$, $u_1$, $p$, $q$ and $f$.
\end{thm}

\begin{proof}
Substitution 
\begin{equation}\label{u-v}
  u(t,x)=\exp{\left\{\frac{1}{2}\int\limits_0^xp(\xi)d\xi\right\}}v(t,x)  
\end{equation}
brings equation \eqref{nonh} to the form
\begin{equation}\label{nonh-s}
    \left\{\begin{array}{l}
    \partial^2_t v(t,x)-\partial^2_x v(t,x)+\left(\frac{p^2(x)}{4}-\frac{p'(x)}{2}+q(x)\right)v(t,x)=g(x)f(t,x),\\
    \qquad\qquad\qquad\qquad\qquad\qquad\qquad\qquad\qquad\qquad (t,x)\in [0,T]\times (0,1),\\
    v(0,x)=g(x)u_0(x),\quad x\in (0,1),\\
    \partial_tv(0,x)=g(x)u_1(x),\quad x\in(0,1),\\
    v(t,0)=0=v(t,1),\quad t\in[0,T].
    \end{array}\right.
\end{equation}

We can use the eigenfunctions \eqref{norm} of the corresponding (homogeneous) eigenvalue problem \eqref{eq-z}-\eqref{bou-z}, and look for a solution in the series form
\begin{equation}\label{nonhu}
    v(t,x)=\sum\limits_{n=1}^\infty v_n(t)\psi_n(x),
\end{equation}
where
$$v_n(t) =\int\limits_0^1 v(t,x)\psi_n(x)dx.$$
We can similarly expand the source function,
\begin{equation}\label{func}
    g(x)f(t,x)=\sum\limits_{n=1}^\infty (gf)_n(t)\psi_n(x),\qquad (gf)_n(t)=\int\limits_0^1g(x)f(t,x)\psi_n(x)dx.
\end{equation}

Now, since we are looking for a twice differentiable function $v(t,x)$ that satisfies the homogeneous
Dirichlet boundary conditions, we can differentiate the Fourier series \eqref{nonhu} term by term and using that the $\psi_n(x)$ satisfies the equation \eqref{eq-z} to obtain
\begin{equation}\label{uxx}
    v_{xx}(t,x) = \sum\limits_{n=1}^\infty v_n(t)\psi''_n(x)=\sum\limits_{n=1}^\infty v_n(t)\left(\frac{p^2(x)}{4}-\frac{p'(x)}{2}+q(x)-\lambda_n\right)\psi_n(x).
\end{equation}

We can also twice differentiate the series \eqref{nonhu} with respect to $t$ to obtain
\begin{equation}\label{utt}
    v_{tt}(t,x) = \sum\limits_{n=1}^\infty v''_n(t)\psi_n(x),
\end{equation}
since the Fourier coefficients of $v_{tt}(t,x)$ are
$$\int\limits_0^1v_{tt}(t,x)\psi_n(x)dx=\frac{\partial^2}{\partial t^2}\left[\int\limits_0^1v(t,x)\psi_n(x)dx\right]=v''_n(t).$$
Differentiation under the above integral is allowed since the resulting integrand is continuous.

Substituting \eqref{utt} and \eqref{uxx} into the equation, and using \eqref{func}, we have
\begin{eqnarray*}
\sum\limits_{n=1}^\infty v''_n(t)\psi_n(x)&-&\sum\limits_{n=1}^\infty v_n(t)\left(\frac{p^2(x)}{4}-\frac{p'(x)}{2}+q(x)-\lambda_n\right)\psi_n(x)\\
&+&\left(\frac{p^2(x)}{4}-\frac{p'(x)}{2}+q(x)\right)\sum\limits_{n=1}^\infty v_n(t)\psi_n(x)=\sum\limits_{n=1}^\infty (gf)_n(t)\psi_n(x),
\end{eqnarray*}
and after a slight rearrangement, we get
$$\sum\limits_{n=1}^\infty \left[v''_n(t)+\lambda_nv_n(t)\right]\psi_n(x)=\sum\limits_{n=1}^\infty (gf)_n(t)\psi_n(x).$$
But then, due to the completeness,
$$v''_n(t)+\lambda_nv_n(t)=(gf)_n(t), \qquad n=1,2,...,$$
which are ODEs for the coefficients $v_n(t)$ of the series \eqref{nonhu}. By the method of variation of constants we get 
\begin{eqnarray*}
v_n(t)&=&A_n\cos \left(\sqrt{\lambda_n}t\right)+\frac{1}{\sqrt{\lambda_n}}B_n \sin \left(\sqrt{\lambda_n}t\right)\\
&-&\frac{1}{\sqrt{\lambda_n}}\cos \left(\sqrt{\lambda_n}t\right)\int\limits_0^t \sin \left(\sqrt{\lambda_n}s\right)(gf)_n(s)ds\\
&+& \frac{1}{\sqrt{\lambda_n}}\sin \left(\sqrt{\lambda_n}t\right)\int\limits_0^t \cos \left(\sqrt{\lambda_n}s\right)(gf)_n(s)ds,
\end{eqnarray*}
where
$$A_n=\int\limits_0^1g(x)u_0(x)\psi_n(x)dx,\qquad B_n=\int\limits_0^1g(x)u_1(x)\psi_n(x)dx.$$
Thus, we can write a solution of the equation \eqref{nonh-s} in the form
\begin{eqnarray*}\label{sol-nh-v}
v(t,x)&=&\sum\limits_{n=0}^\infty\left[A_n\cos\left(\sqrt{\lambda_n}t\right)+\frac{1}{\sqrt{\lambda_n}}B_n\sin \left(\sqrt{\lambda_n}t\right)\right]\psi_n(x) \nonumber\\
&-& \sum\limits_{n=1}^\infty\frac{1}{\sqrt{\lambda_n}}\cos\left(\sqrt{\lambda_n}t\right)\int\limits_0^t\sin\left(\sqrt{\lambda_n}s\right)(gf)_n(s)ds\psi_n(x)\nonumber\\
&+&\sum\limits_{n=1}^\infty\frac{1}{\sqrt{\lambda_n}}\sin \left(\sqrt{\lambda_n}t\right)\int\limits_0^t \cos \left(\sqrt{\lambda_n}s\right)(gf)_n(s)ds\psi_n(x).
\end{eqnarray*}

According to \eqref{u-v}, we obtain the solution of the equation \eqref{nonh} in the following form
\begin{eqnarray}\label{sol-nh}
u(t,x)&=&\sum\limits_{n=0}^\infty\left[A_n\cos\left(\sqrt{\lambda_n}t\right)+\frac{1}{\sqrt{\lambda_n}}B_n\sin\left(\sqrt{\lambda_n}t\right)\right]\phi_n(x) \nonumber\\
&-& \sum\limits_{n=1}^\infty\frac{1}{\sqrt{\lambda_n}}\cos\left(\sqrt{\lambda_n}t\right)\int\limits_0^t\sin\left(\sqrt{\lambda_n}s\right)(gf)_n(s)ds\phi_n(x)\nonumber\\
&+&\sum\limits_{n=1}^\infty\frac{1}{\sqrt{\lambda_n}}\sin\left(\sqrt{\lambda_n}t\right)\int\limits_0^t \cos \left(\sqrt{\lambda_n}s\right)(gf)_n(s)ds\phi_n(x).
\end{eqnarray}

Let us estimate $\|u(t,\cdot)\|_{L^2}$. For this we use the estimates
\begin{eqnarray}\label{est-nonh}
\int\limits_0^1|u(t,x)|^2dx&\lesssim&\int\limits_0^1\left|\sum\limits_{n=0}^\infty\left[A_n\cos\left(\sqrt{\lambda_n}t\right)+\frac{1}{\sqrt{\lambda_n}}B_n\sin \left(\sqrt{\lambda_n}t\right)\right]\phi_n(x)\right|^2dx \nonumber\\
&+& \int\limits_0^1\left|\sum\limits_{n=1}^\infty\frac{1}{\sqrt{\lambda_n}}\cos\left(\sqrt{\lambda_n}t\right)\int\limits_0^t\sin\left(\sqrt{\lambda_n}s\right)(gf)_n(s)ds\phi_n(x)\right|^2dx\nonumber\\
&+&\int\limits_0^1\left|\sum\limits_{n=1}^\infty\frac{1}{\sqrt{\lambda_n}}\sin \left(\sqrt{\lambda_n}t\right)\int\limits_0^t \cos \left(\sqrt{\lambda_n}s\right)(gf)_n(s)ds\phi_n(x)\right|^2dx\nonumber\\
&=&I_1+I_2+I_3.
\end{eqnarray}
For $I_1$ by using \eqref{est1} for the homogeneous case we have that
\begin{eqnarray*}
I_1&:=&\int\limits_0^1\left|\sum\limits_{n=0}^\infty\left[A_n\cos\left(\sqrt{\lambda_n}t\right)+\frac{1}{\sqrt{\lambda_n}}B_n\sin \left(\sqrt{\lambda_n}t\right)\right]\phi_n(x)\right|^2dx\\
&\lesssim& \exp{\left\{\|p\|_{L^1}\right\}}\left(\|gu_0\|^2_{L^2}+\|gu_1\|^2_{W^{-1}_{\mathcal{L}}}\right).
\end{eqnarray*}
Now we estimate $I_2$ in \eqref{est-nonh} as
\begin{eqnarray}\label{I2+}
I_2&:=&\int\limits_0^1\left|\sum\limits_{n=1}^\infty\frac{1}{\sqrt{\lambda_n}}\cos\left(\sqrt{\lambda_n}t\right)\int\limits_0^t\sin\left(\sqrt{\lambda_n}s\right)(gf)_n(s)ds\phi_n(x)\right|^2dx\nonumber\\
&\lesssim& \exp{\left\{\|p\|_{L^1}\right\}}\sum\limits_{n=1}^\infty \left[\int\limits_0^t|(gf)_n(s)|ds\right]^2.
\end{eqnarray}
Using Holder's inequality and taking into account that $t\in [0,T]$ we get
\begin{equation}\label{I2-1}
    \left[\int\limits_0^t|(gf)_n(s)|ds\right]^2\leq \left[\int\limits_0^T 1\cdot| (gf)_n(t)|dt\right]^2\leq T\int\limits_0^T| (gf)_n(t)|^2dt,
\end{equation}
since $(gf)_n(t)$ is the Fourier's coefficient of the function $g(x)f(t,x)$ and by Parseval's identity we obtain
\begin{eqnarray}\label{I2-2}
\sum\limits_{n=1}^\infty T\int\limits_0^T|(gf)_n(t)|^2dt&=& T\int\limits_0^T\sum\limits_{n=1}^\infty|(gf)_n(t)|^2dt = T\int\limits_0^T\|gf(t,\cdot)\|^2_{L^2}dt\nonumber\\
&\leq& T\|g\|^2_{L^\infty}\int\limits_0^T\|f(t,\cdot)\|^2_{L^2}dt.
\end{eqnarray}
Since
$$\|f\|_{C([0,T],L^2(0,1))}=\max\limits_{0\leq t\leq T}\|f(t,\cdot)\|_{L^2},$$
we arrive at the inequality
\begin{equation}\label{I2-3}
    T\int\limits_0^T\|f(t,\cdot)\|^2_{L^2}dt\leq T^2\|f\|^2_{C([0,T],L^2(0,1))}.
\end{equation}
Thus,
\begin{eqnarray}\label{I2}
I_2&=&\int\limits_0^1\left|\sum\limits_{n=1}^\infty\frac{1}{\sqrt{\lambda_n}}\cos\left(\sqrt{\lambda_n}t\right)\int\limits_0^t\sin\left(\sqrt{\lambda_n}s\right)(gf)_n(s)ds\phi_n(x)\right|^2dx\nonumber\\
&\lesssim&  T^2\exp{\left\{\|p\|_{L^1}\right\}}\|g\|^2_{L^\infty}\|f\|^2_{C([0,T],L^2(0,1))},
\end{eqnarray}
and $I_3$ in \eqref{est-nonh} is evaluated similarly
\begin{eqnarray}\label{I3}
    I_3&:=&\int\limits_0^1\left|\sum\limits_{n=1}^\infty\frac{1}{\sqrt{\lambda_n}}\sin \left(\sqrt{\lambda_n}t\right)\int\limits_0^t \cos \left(\sqrt{\lambda_n}s\right)(gf)_n(s)ds\phi_n(x)\right|^2dx \nonumber\\
    &\lesssim& T^2\exp{\left\{\|p\|_{L^1}\right\}}\|g\|^2_{L^\infty}\|f\|^2_{C([0,T],L^2(0,1))}.
\end{eqnarray}

We finally get
$$
    \|u(t,\cdot)\|^2_{L^2}\lesssim \exp{\left\{\|p\|_{L^1}\right\}}\left(\|gu_0\|^2_{L^2}+\|gu_1\|^2_{W^{-1}_\mathcal{L}}+2T^2\|g\|^2_{L^\infty}\|f\|^2_{C([0,T],L^2(0,1))}\right).
$$
Let us estimate $\|\partial_tu(t,\cdot)\|_{L^2}$, for this we calculate $\partial_tu(t,x)$ as follows
\begin{eqnarray*}
\partial_tu(t,x)&=&\sum\limits_{n=0}^\infty\left[-\sqrt{\lambda_n}A_n\sin\left(\sqrt{\lambda_n}t\right)+B_n\cos \left(\sqrt{\lambda_n}t\right)\right]\phi_n(x) \\
&+& \sum\limits_{n=1}^\infty\sin\left(\sqrt{\lambda_n}t\right)\int\limits_0^t\sin\left(\sqrt{\lambda_n}s\right)(gf)_n(s)ds\phi_n(x)\\
&+&
\sum\limits_{n=1}^\infty\cos \left(\sqrt{\lambda_n}t\right)\int\limits_0^t \cos \left(\sqrt{\lambda_n}s\right)(gf)_n(s)ds\phi_n(x),
\end{eqnarray*}
then
\begin{eqnarray*}
\|\partial_tu(t,\cdot)\|^2_{L^2}&=&\int\limits_0^1|\partial_tu(t,x)|^2dx\\
&\lesssim&\int\limits_0^1\left|\sum\limits_{n=0}^\infty\left[-\sqrt{\lambda_n}A_n\sin\left(\sqrt{\lambda_n}t\right)+B_n\cos \left(\sqrt{\lambda_n}t\right)\right]\phi_n(x)\right|^2dx \\
&+& \int\limits_0^1\left|\sum\limits_{n=1}^\infty\sin\left(\sqrt{\lambda_n}t\right)\int\limits_0^t\sin\left(\sqrt{\lambda_n}s\right)(gf)_n(s)ds\phi_n(x)\right|^2dx\\
&+&
\int\limits_0^1\left|\sum\limits_{n=1}^\infty\cos \left(\sqrt{\lambda_n}t\right)\int\limits_0^t \cos \left(\sqrt{\lambda_n}s\right)(gf)_n(s)ds\phi_n(x)\right|^2dx\\
&\lesssim&\sum\limits_{n=0}^\infty\left|-\sqrt{\lambda_n}A_n\sin\left(\sqrt{\lambda_n}t\right)+B_n\cos \left(\sqrt{\lambda_n}t\right)\right|^2 \\
&+& \sum\limits_{n=1}^\infty\left|\sin\left(\sqrt{\lambda_n}t\right)\int\limits_0^t\sin\left(\sqrt{\lambda_n}s\right)(gf)_n(s)ds\right|^2\\
&+&\sum\limits_{n=1}^\infty\left|\cos \left(\sqrt{\lambda_n}t\right)\int\limits_0^t \cos \left(\sqrt{\lambda_n}s\right)(gf)_n(s)ds\right|^2,
\end{eqnarray*}
by using the \eqref{est2} for the homogeneous case and conducting evaluations as in \eqref{I2}, \eqref{I3} we obtain
$$\|\partial_tu(t,\cdot)\|^2_{L^2}\lesssim \exp{\{\|p\|_{L^1}\}}\left(\|gu_0\|^2_{W^1_\mathcal{L}}+\|gu_1\|^2_{L^2}+2T^2\|g\|^2_{L^\infty}\|f\|^2_{C([0,T],L^2(0,1))}\right).$$
 
For \eqref{es-nh3} we get
\begin{eqnarray}\label{nh-ux}
\|\partial_x u(t,\cdot)\|^2_{L^2}&=&\int\limits_0^1|\partial_xu(t,x)|^2dt\nonumber\\
&\lesssim&
\int\limits_0^1\left|\sum\limits_{n=1}^\infty\left[A_n\cos\left(\sqrt{\lambda_n}t\right)+\frac{1}{\sqrt{\lambda_n}}B_n\sin\left(\sqrt{\lambda_n}t\right)\right]\phi'_n(x)\right|^2dx\nonumber\\
&+& \int\limits_0^1\left|\sum\limits_{n=1}^\infty\frac{1}{\sqrt{\lambda_n}}\cos\left(\sqrt{\lambda_n}t\right)\int\limits_0^t\sin\left(\sqrt{\lambda_n}s\right)(gf)_n(s)ds\phi'_n(x)\right|^2dx\nonumber\\
&+&\int\limits_0^1\left|\sum\limits_{n=1}^\infty\frac{1}{\sqrt{\lambda_n}}\sin\left(\sqrt{\lambda_n}t\right)\int\limits_0^t \cos \left(\sqrt{\lambda_n}s\right)(gf)_n(s)ds\phi'_n(x)\right|^2dx\nonumber\\
&=&K_1+K_2+K_3.
\end{eqnarray}
Taking \eqref{u_x} into account,  we have that
\begin{eqnarray*}
K_1&:=&\int\limits_0^1\left|\sum\limits_{n=1}^\infty\left[A_n\cos\left(\sqrt{\lambda_n}t\right)+\frac{1}{\sqrt{\lambda_n}}B_n\sin\left(\sqrt{\lambda_n}t\right)\right]\phi'_n(x)\right|^2dx\\
&\lesssim&\exp{\left\{\|p\|_{L^1}\right\}}\left\{\left(1+\|\nu\|^2_{L^2}\left(\|\nu\|^2_{L^2}+\|p\|^2_{L^2}+\|p'\|^2_{L^1}\right)\right)\left(\|gu_0\|^2_{W^1_{\mathcal{L}}}+\|gu_1\|^2_{L^2}\right)\right.\nonumber\\
&+& \left.\left(\|p\|^2_{L^\infty}+\|\nu\|^2_{L^\infty}\right)\left(\|gu_0\|^2_{L^2}+\|gu_1\|^2_{W^{-1}_{\mathcal{L}}}\right)\right\}. 
\end{eqnarray*}
For $K_2$ in \eqref{nh-ux} using \eqref{phi'}, \eqref{38}, \eqref{39} and \eqref{40} we obtain
\begin{eqnarray}\label{K2}
K_2&:=&\int\limits_0^1\left|\sum\limits_{n=1}^\infty\frac{1}{\sqrt{\lambda_n}}\cos\left(\sqrt{\lambda_n}t\right)\int\limits_0^t\sin\left(\sqrt{\lambda_n}s\right)(gf)_n(s)ds\phi'_n(x)\right|^2dx\nonumber\\
&\lesssim& \exp{\left\{\|p\|_{L^1}\right\}}\left(1+\|\nu\|^2_{L^2}\left(\|\nu\|^2_{L^2}+\|p\|^2_{L^2}+\|p'\|^2_{L^1}\right)\right)\times\nonumber\\
&\times&\sum\limits_{n=1}^\infty\left|\cos\left(\sqrt{\lambda_n}t\right)\int\limits_0^t\sin\left(\sqrt{\lambda_n}s\right)(gf)_n(s)ds\right|^2\nonumber\\
&+&\exp{\left\{\|p\|_{L^1}\right\}}\left(\|p\|^2_{L^\infty}+\|\nu\|^2_{L^\infty}\right)\times\nonumber\\
&\times&\sum\limits_{n=1}^\infty\left|\frac{1}{\sqrt{\lambda_n}}\cos\left(\sqrt{\lambda_n}t\right)\int\limits_0^t\sin\left(\sqrt{\lambda_n}s\right)(gf)_n(s)ds\right|^2,
\end{eqnarray}
where
\begin{equation*}
\begin{array}{l}
\sum\limits_{n=1}^\infty\left|\frac{1}{\sqrt{\lambda_n}}\cos\left(\sqrt{\lambda_n}t\right)\int\limits_0^t\sin\left(\sqrt{\lambda_n}s\right)(gf)_n(s)ds\right|^2\\
\leq \sum\limits_{n=1}^\infty\left|\cos\left(\sqrt{\lambda_n}t\right)\int\limits_0^t\sin\left(\sqrt{\lambda_n}s\right)(gf)_n(s)ds\right|^2,
\end{array}
\end{equation*}
since $\lambda_n\geq 1,\,n=1,2,...,$ according to \eqref{I2}, so it is enough to estimate 
\begin{equation}\label{f2}
    \left|\sum\limits_{n=1}^\infty\cos\left(\sqrt{\lambda_n}t\right)\int\limits_0^t\sin\left(\sqrt{\lambda_n}s\right)(gf)_n(s)ds\right|^2 \lesssim T^2\|g\|^2_{L^\infty}\|f\|^2_{C([0,T],L^2(0,1))}.
\end{equation}

Thus,
\begin{eqnarray*}
K_2&\lesssim&\exp{\left\{\|p\|_{L^1}\right\}}\left(1+\|\nu\|^2_{L^2}\left(\|\nu\|^2_{L^2}+\|p\|^2_{L^2}+\|p'\|^2_{L^1}\right)+\|p\|^2_{L^\infty}+\|\nu\|^2_{L^\infty}\right)\times\\
&\times&T^2\|g\|^2_{L^\infty}\|f\|^2_{C([0,T],L^2(0,1))}.
\end{eqnarray*}
For $K_3$ in \eqref{nh-ux} we similarly get
\begin{eqnarray*}
K_3&\lesssim&\exp{\left\{\|p\|_{L^1}\right\}}\left(1+\|\nu\|^2_{L^2}\left(\|\nu\|^2_{L^2}+\|p\|^2_{L^2}+\|p'\|^2_{L^1}\right)+\|p\|^2_{L^\infty}+\|\nu\|^2_{L^\infty}\right)\times\\
&\times&T^2\|g\|^2_{L^\infty}\|f\|^2_{C([0,T],L^2(0,1))}.
\end{eqnarray*}
Taking into account the estimates for $K_1$, $K_2$ and $K_3$, we obtain
\begin{eqnarray*}
\|\partial_x u(t,\cdot)\|^2_{L^2}&\lesssim&\exp{\left\{\|p\|_{L^1}\right\}}\left\{\left(1+\|\nu\|^2_{L^2}\left(\|\nu\|^2_{L^2}+\|p\|^2_{L^2}+\|p'\|^2_{L^1}\right)\right)\times\right.\nonumber\\
&\times&\left(\|gu_0\|^2_{W^1_{\mathcal{L}}}+\|gu_1\|^2_{L^2}\right)+ \left(\|p\|^2_{L^\infty}+\|\nu\|^2_{L^\infty}\right)\left(\|gu_0\|^2_{L^2}+\|gu_1\|^2_{W^{-1}_{\mathcal{L}}}\right)\\
&+&\left(1+\|\nu\|^2_{L^2}\left(\|\nu\|^2_{L^2}+\|p\|^2_{L^2}+\|p'\|^2_{L^1}\right)+\|p\|^2_{L^\infty}+\|\nu\|^2_{L^\infty}\right)\times\\
&\times&\left.2T^2\|g\|^2_{L^\infty}\|f\|^2_{C([0,T],L^2(0,1))}\right\}.
\end{eqnarray*}

We have  $\phi_n''(x)=p(x)\phi'_n(x)+(q(x)-\lambda_n)\phi_n(x)$, so that
\begin{eqnarray*}
\|\partial^2_xu(t,\cdot)\|^2_{L^2}&=&\int\limits_0^1|\partial^2_xu(t,x)|^2dx\\
&\lesssim& \int\limits_0^1\left|\sum\limits_{n=0}^\infty\left[A_n\cos\left(\sqrt{\lambda_n}t\right)+\frac{1}{\sqrt{\lambda_n}}B_n\sin\left(\sqrt{\lambda_n}t\right)\right]\times\right.\\
&\times&\left(p(x)\phi'_n(x)+(q(x)-\lambda_n\right)\phi_n(x))\Biggr|^2dx \\
&+& \int\limits_0^1\left|\sum\limits_{n=1}^\infty\frac{1}{\sqrt{\lambda_n}}\cos\left(\sqrt{\lambda_n}t\right)\int\limits_0^t\sin\left(\sqrt{\lambda_n}s\right)(gf)_n(s)ds\times\right.\\
&\times&\left(p(x)\phi'_n(x)+(q(x)-\lambda_n\right)\phi_n(x))\Biggr|^2dx \\
&+&\int\limits_0^1\left|\sum\limits_{n=1}^\infty\frac{1}{\sqrt{\lambda_n}}\sin\left(\sqrt{\lambda_n}t\right)\int\limits_0^t \cos \left(\sqrt{\lambda_n}s\right)(gf)_n(s)ds\times\right.\\
&\times&\left(p(x)\phi'_n(x)+(q(x)-\lambda_n\right)\phi_n(x))\Biggr|^2dx \\
&=& E_1+E_2+E_3.
\end{eqnarray*}
Using \eqref{est4} we get
\begin{eqnarray*}
E_1&:=&\int\limits_0^1\left|\sum\limits_{n=0}^\infty\left[A_n\cos\left(\sqrt{\lambda_n}t\right)+\frac{1}{\sqrt{\lambda_n}}B_n\sin\left(\sqrt{\lambda_n}t\right)\right]\times\right.\\
&\times&(p(x)\phi'_n(x)+(q(x)-\lambda_n)\phi_n(x))\Bigg|^2dx\\&\lesssim&
\exp{\{\|p\|_{L^2}\}}\Big\{\|p\|^2_{L^\infty}\left(\left(1+\|\nu\|^2_{L^2}\left(\|\nu\|^2_{L^2}+\|p\|^2_{L^2}+\|p'\|^2_{L^1}\right)\right)\times\right.\nonumber\\
&\times& \left.\left(\|gu_0\|^2_{W^1_{\mathcal{L}}}+\|gu_1\|^2_{L^2}\right)+\left(\|p\|^2_{L^\infty}+\|\nu\|^2_{L^\infty}\right)\left(\|gu_0\|^2_{L^2}+\|gu_1\|^2_{W^{-1}_{\mathcal{L}}}\right)\right)\nonumber\\
&+&\left.\|q\|^2_{L^\infty} \left(\|gu_0\|^2_{L^2}+\|gu_1\|^2_{W^{-1}_{\mathcal{L}}}\right)+\left\|gu_0\right\|^2_{W^2_{\mathcal{L}}}+\|gu_1\|^2_{W^1_{\mathcal{L}}}\right\}.
\end{eqnarray*}
Let us estimate $E_2$ by using \eqref{K2}-\eqref{f2}
\begin{eqnarray*}
E_2&:=&\int\limits_0^1\left|\sum\limits_{n=1}^\infty\frac{1}{\sqrt{\lambda_n}}\cos\left(\sqrt{\lambda_n}t\right)\int\limits_0^t\sin\left(\sqrt{\lambda_n}s\right)(gf)_n(s)ds\times\right.\\
&\times&(p(x)\phi'_n(x)+(q(x)-\lambda_n)\phi_n(x))\Biggr|^2dx\\
&\lesssim& \exp{\{\|p\|_{L^1}\}}\Big\{\left(\|p\|^2_{L^\infty}\left(1+\|\nu\|^2_{L^2}\left(\|\nu\|^2_{L^2}+\|p\|^2_{L^2}+\|p'\|^2_{L^1}\right)+\|p\|^2_{L^\infty}+\|\nu\|^2_{L^\infty}\right)\right.\\
&+&\left.\left.\|q\|^2_{L^\infty}\right)T^2\|g\|^2_{L^\infty}\|f\|^2_{C([0,T],L^2(0,1))}+T^2\|g\|^2_{L^\infty}\|f\|^2_{C^1([0,T],L^2(0,1))}\right\}.
\end{eqnarray*}
We similarly get
\begin{eqnarray*}
E_3&:=&\int\limits_0^1\left|\sum\limits_{n=1}^\infty\frac{1}{\sqrt{\lambda_n}}\sin\left(\sqrt{\lambda_n}t\right)\int\limits_0^t \cos \left(\sqrt{\lambda_n}s\right)(gf)_n(s)ds\times\right.\\
&\times&(p(x)\phi'_n(x)+(q(x)-\lambda_n)\phi_n(x))\Biggr|^2dx\\
&\lesssim& \exp{\{\|p\|_{L^1}\}}\Big\{\left(\|p\|^2_{L^\infty}\left(1+\|\nu\|^2_{L^2}\left(\|\nu\|^2_{L^2}+\|p\|^2_{L^2}+\|p'\|^2_{L^1}\right)+\|p\|^2_{L^\infty}+\|\nu\|^2_{L^\infty}\right)\right.\\
&+&\left.\left.\|q\|^2_{L^\infty}\right)T^2\|g\|^2_{L^\infty}\|f\|^2_{C([0,T],L^2(0,1))}+T^2\|g\|^2_{L^\infty}\|f\|^2_{C^1([0,T],L^2(0,1))}\right\}.
\end{eqnarray*}
Therefore,
\begin{eqnarray*}
\|\partial^2_xu(t,\cdot)\|^2_{L^2}&\lesssim& 
\exp{\{\|p\|_{L^2}\}}\|p\|^2_{L^\infty}\Big\{\left(1+\|\nu\|^2_{L^2}\left(\|\nu\|^2_{L^2}+\|p\|^2_{L^2}+\|p'\|^2_{L^1}\right)\right)\times\nonumber\\
&\times& \left(\|gu_0\|^2_{W^1_{\mathcal{L}}}+\|gu_1\|^2_{L^2}\right)+\left(\|p\|^2_{L^\infty}+\|\nu\|^2_{L^\infty}\right)\left(\|gu_0\|^2_{L^2}+\|gu_1\|^2_{W^{-1}_{\mathcal{L}}}\right)\nonumber\\
&+&\|q\|^2_{L^\infty} \left(\|gu_0\|^2_{L^2}+\|gu_1\|^2_{W^{-1}_{\mathcal{L}}}\right)+\left\|gu_0\right\|^2_{W^2_{\mathcal{L}}}+\|gu_1\|^2_{W^1_{\mathcal{L}}}\nonumber\\
&+& \left(\|p\|^2_{L^\infty}\left(1+\|\nu\|^2_{L^2}\left(\|\nu\|^2_{L^2}+\|p\|^2_{L^2}+\|p'\|^2_{L^1}\right)+\|p\|^2_{L^\infty}+\|\nu\|^2_{L^\infty}\right)\right.\\
&+&\left.\left.\|q\|^2_{L^\infty}\right)2T^2\|g\|^2_{L^\infty}\|f\|^2_{C([0,T],L^2(0,1))}+T^2\|g\|^2_{L^\infty}\|f\|^2_{C^1([0,T],L^2(0,1))}\right\}.
\end{eqnarray*}

Let us estimate $\|u(t,\cdot)\|^2_{W^k_{\mathcal{L}}}$:
\begin{eqnarray*}
\|u(t,\cdot)\|^2_{W^k_{\mathcal{L}}}&\lesssim&
\int\limits_0^1\left|\sum\limits_{n=0}^\infty\left[A_n\cos\left(\sqrt{\lambda_n}t\right)+\frac{1}{\sqrt{\lambda_n}}B_n\sin \left(\sqrt{\lambda_n}t\right)\right]\mathcal{L}^\frac{k}{2}\phi_n(x)\right|^2dx \nonumber\\
&+& \int\limits_0^1\left|\sum\limits_{n=1}^\infty\frac{1}{\sqrt{\lambda_n}}\cos\left(\sqrt{\lambda_n}t\right)\int\limits_0^t\sin\left(\sqrt{\lambda_n}s\right)(gf)_n(s)ds\mathcal{L}^\frac{k}{2}\phi_n(x)\right|^2dx\nonumber\\
&+&\int\limits_0^1\left|\sum\limits_{n=1}^\infty\frac{1}{\sqrt{\lambda_n}}\sin \left(\sqrt{\lambda_n}t\right)\int\limits_0^t \cos \left(\sqrt{\lambda_n}s\right)(gf)_n(s)ds\mathcal{L}^\frac{k}{2}\phi_n(x)\right|^2dx\nonumber\\
&=&F_1+F_2+F_3.
\end{eqnarray*}
By using \eqref{est5} we have
\begin{eqnarray*}
F_1&:=&\int\limits_0^1\left|\sum\limits_{n=0}^\infty\left[A_n\cos\left(\sqrt{\lambda_n}t\right)+\frac{1}{\sqrt{\lambda_n}}B_n\sin \left(\sqrt{\lambda_n}t\right)\right]\mathcal{L}^\frac{k}{2}\phi_n(x)\right|^2dx\\
&\lesssim&\exp{\left\{\|p\|_{L^1}\right\}}\left(\left\|gu_0\right\|^2_{W^k_{\mathcal{L}}}+\left\|gu_1\right\|^2_{W^{k-1}_{\mathcal{L}}}\right).
\end{eqnarray*}
Using that $\mathcal{L}^{\frac{k}{2}}\phi_n(x)=\lambda ^{\frac{k}{2}}\phi_n(x)$,  \eqref{norm-phi} and following as \eqref{I2+}-\eqref{I2-3}, we obtain
\begin{eqnarray*}
F_2&:=& \int\limits_0^1\left|\sum\limits_{n=1}^\infty\frac{1}{\sqrt{\lambda_n}}\cos\left(\sqrt{\lambda_n}t\right)\int\limits_0^t\sin\left(\sqrt{\lambda_n}s\right)(gf)_n(s)ds\mathcal{L}^\frac{k}{2}\phi_n(x)\right|^2dx\\
&\lesssim&\exp{\|p\|_{L^1}}\int\limits_0^1\left|\sum\limits_{n=1}^\infty\int\limits_0^t\sin\left(\sqrt{\lambda_n}s\right)(gf)_n(s)ds\lambda^\frac{k-1}{2}\psi_n(x)\right|^2dx\\
&\lesssim&T\exp{\|p\|_{L^1}}\int\limits_0^T\sum\limits_{n=1}^\infty\left|\lambda^\frac{k-1}{2}(gf)_n(t)\right|^2dt=T\exp{\|p\|_{L^1}}\int\limits_0^T\left\|\lambda^\frac{k-1}{2}gf(t,\cdot)\right\|^2_{L^2}dt\\
&=&T\exp{\|p\|_{L^1}}\int\limits_0^T\left\|\mathcal{L}^\frac{k-1}{2}gf(t,\cdot)\right\|^2_{L^2}dt=T\exp{\|p\|_{L^1}}\int\limits_0^T\left\|gf(t,\cdot)\right\|^2_{\mathcal{L}^{k-1}}dt\\
&\leq& T^2\exp{\|p\|_{L^1}}\left\|gf(\cdot,\cdot)\right\|^2_{C([0,T],W^{k-1}_\mathcal{L}(0,1))}.
\end{eqnarray*}
We similarly get
\begin{eqnarray*}
F_3&:=&\int\limits_0^1\left|\sum\limits_{n=1}^\infty\frac{1}{\sqrt{\lambda_n}}\sin \left(\sqrt{\lambda_n}t\right)\int\limits_0^t \cos \left(\sqrt{\lambda_n}s\right)(gf)_n(s)ds\mathcal{L}^\frac{k}{2}\phi_n(x)\right|^2dx\nonumber\\
&\lesssim&T^2\exp{\|p\|_{L^1}}\left\|gf(\cdot,\cdot)\right\|^2_{C([0,T],W^{k-1}_\mathcal{L}(0,1))}.
\end{eqnarray*}
Thus, 
\begin{eqnarray*}
\|u(t,\cdot)\|^2_{W^k_{\mathcal{L}}}\lesssim \exp{\left\{\|p\|_{L^1}\right\}}\left(\left\|gu_0\right\|^2_{W^k_{\mathcal{L}}}+\left\|gu_1\right\|^2_{W^{k-1}_{\mathcal{L}}}+2T^2\left\|gf(\cdot,\cdot)\right\|^2_{C([0,T],W^{k-1}_\mathcal{L}(0,1))}\right).
\end{eqnarray*}

The proof of Theorem \ref{non-hom} is complete.
\end{proof}

We will now express all the estimates in terms of the coefficients, to be used in the very weak well-posedness in Section \ref{ch4}.

\begin{cor}\label{cor2}
Assume that $p'\in L^2(0,1)$, $q=\nu'$, $\nu \in L^\infty(0,1)$ and $f(t,x)\in C^1([0,T],L^2(0,1))$. If the initial data satisfy $(u_0,\, u_1) \in L^2(0,1)$ and $(u_0'',\, u''_1)\in L^2(0,1)$, then the non-homogeneous wave equation with initial/boundary conditions  \eqref{nonh} has unique solution $u\in C([0,T], L^2(0,1))$ such that
\begin{equation}\label{ec-nh1}
    \|u(t,\cdot)\|^2_{L^2}\lesssim \exp{\{2\|p\|_{L^2}\}}\left(\|u_0\|^2_{L^2}+\|u_1\|^2_{L^2}+2T^2\|f\|_{C([0,1],L^2(0,1))}\right),
\end{equation}
\begin{eqnarray}\label{ec-nh2}
\|\partial_t u(t,\cdot)\|^2_{L^2}&\lesssim& \exp{\{2\|p\|_{L^1}\}}\left\{ \|u''_0\|^2_{L^2}+\|p\|^2_{L^\infty}\|u'_0\|^2_{L^2}+\left(\|p\|^4_{L^\infty}+\|p'\|^2_{L^\infty}\right.\right.\nonumber\\
&+&\left.\left.\|q\|^2_{L^\infty}\right)\|u_0\|^2_{L^2}+\|u_1\|^2_{L^2}+2T^2\|g\|^2_{L^\infty}\|f\|_{C([0,1],L^2(0,1))}\right\},
\end{eqnarray}
\begin{eqnarray}\label{ec-nh3}
\|\partial_xu(t,\cdot)\|^2_{L^2}&\lesssim& \exp{\left\{2\|p\|_{L^1}\right\}}\left\{\left(1+\|\nu\|^2_{L^2}\left(\|\nu\|^2_{L^2}+\|p\|^2_{L^2}+\|p'\|^2_{L^1}\right)\right)\times\right.\nonumber\\
&\times& \left(\|u''_0\|^2_{L^2}+\|p\|^2_{L^\infty}\|u'_0\|^2_{L^2}+\left(\|p\|^4_{L^\infty}+\|p'\|^2_{L^\infty}+\|q\|^2_{L^\infty}\right)\|u_0\|^2_{L^2}\right.\nonumber\\
&+&\left.\|u_1\|^2_{L^2}\right)+\left(\|p\|^2_{L^\infty}+\|\nu\|^2_{L^\infty}\right)\left(\|u_0\|^2_{L^2}+\|u_1\|^2_{L^2}\right)\nonumber\\
&+&\left(1+\|\nu\|^2_{L^2}\left(\|\nu\|^2_{L^2}+\|p\|^2_{L^2}+\|p'\|^2_{L^1}\right)+\|p\|^2_{L^\infty}+\|\nu\|^2_{L^\infty}\right)\times\nonumber\\
&\times&\left.2T^2\|f\|^2_{C([0,T],L^2(0,1))}\right\},
\end{eqnarray}
\begin{eqnarray}\label{ec-nh4}
\|\partial^2_xu(t,\cdot)\|^2_{L^2}&\lesssim&\exp{\left\{2\|p\|_{L^1}\right\}}\Big\{\left(1+\|\nu\|^2_{L^2}\left(\|\nu\|^2_{L^2}+\|p\|^2_{L^2}+\|p'\|^2_{L^1}\right)\right)\times\nonumber\\
&\times& \left(\|u''_0\|^2_{L^2}+\|p\|^2_{L^\infty}\|u'_0\|^2_{L^2}+\left(\|p\|^4_{L^\infty}+\|p'\|^2_{L^\infty}+\|q\|^2_{L^\infty}\right)\|u_0\|^2_{L^2}\right.\nonumber\\
&+&\left.\|u_1\|^2_{L^2}\right)+\left(\|p\|^2_{L^\infty}+\|\nu\|^2_{L^\infty}\right)\left(\|u_0\|^2_{L^2}+\|u_1\|^2_{L^2}\right)\nonumber\\
&+&\|u''_0\|^2_{L^2}+\|u''_1\|^2_{L^2}+\|p\|^2_{L^\infty}\left(\|u'_0\|^2_{L^2}+\|u'_1\|^2_{L^2}\right)\nonumber\\
&+&\left(\|p\|^4_{L^\infty}+\|p'\|^2_{L^\infty}+\|q\|^2_{L^\infty}\right)\left(\|u_0\|^2_{L^2}+\|u_1\|^2_{L^2}\right)\nonumber\\
&+&\left(\|p\|^2_{L^\infty}\left(1+\|\nu\|^2_{L^2}\left(\|\nu\|^2_{L^2}+\|p\|^2_{L^2}+\|p'\|^2_{L^1}\right)+\|p\|^2_{L^\infty}+\|\nu\|^2_{L^\infty}\right)\right.\nonumber\\
&+&\left.\left.\|q\|^2_{L^\infty}\right)2T^2\|f\|^2_{C([0,T],L^2(0,1))}+T^2\|f\|^2_{C^1([0,T],L^2(0,1))}\right\},
\end{eqnarray}
where the constants in these inequalities are independent of $u_0$, $u_1$, $p$, $q$ and $f$.
\end{cor}
The proof of Corollary \ref{cor2} immediately follows from Corollary \ref{cor1} and Theorem \ref{non-hom}.

\section{Very weak solutions}\label{ch4}
In this section we will analyse the solutions for less regular potentials $q$ and $p$. For this we will be using the notion of very weak solutions. 

Assume that the coefficients $q$, $p$ and initial data $(u_0,\, u_1)$ are the distributions on $(0,1)$. 

\begin{defn}\label{D1} (i)
A net of functions $\left(u_\varepsilon=u_\varepsilon(t,x)\right)$ is said to be $L^2$-moderate if there exist $N\in \mathbb{N}_0$ and $C>0$ such that 
$$\|u_\varepsilon(t,\cdot)\|_{L^2}\leq C \varepsilon^{-N}, \quad \text{for all } t\in[0,T].$$
(ii) Moderateness of data: a net of functions $(u_{0,\varepsilon}=u_{0,\varepsilon}(x))$ is said to be $H^2$-moderate if there exist $N\in \mathbb{N}_0$ and $C>0$ such that 
$$\|u'_{0,\varepsilon}\|_{L^2}\leq C\varepsilon^{-N}, \quad \|u''_{0,\varepsilon}\|_{L^2}\leq C\varepsilon^{-N}.$$
\end{defn}

\begin{defn}\label{D1-1} (i)
A net of functions $\left(\nu_\varepsilon=\nu_\varepsilon(x)\right)$ is said to be $L^\infty_1$-moderate if there exist $N\in \mathbb{N}_0$ and $C>0$ such that 
\begin{equation}\label{nu-mod}
    \|\nu_\varepsilon\|_{L^\infty}\leq C \varepsilon^{-N}, \qquad \|\nu'_\varepsilon\|_{L^\infty}\leq C \varepsilon^{-N}.
\end{equation}
(ii) A net of functions $(p_\varepsilon)$ is said to be $\log$-$L^\infty_1$-moderate if there exist $N\in \mathbb{N}_0$ and $C>0$ such that
$$\|p_\varepsilon\|_{L^\infty}\leq C |\log \varepsilon|^N, \qquad \|p'_\varepsilon\|_{L^\infty}\leq C \varepsilon^{-N}.$$
\end{defn}
\begin{rem}
We note that for the clarity of expression, we put two condition in \eqref{nu-mod} explicitly. However, we note that the first one follows from the second:
$$\left|\nu_\varepsilon(x)\right|=\left|\int\limits_0^x \nu'_\varepsilon(\xi)d\xi\right|\leq C\|\nu'_\varepsilon\|_{L^\infty(0,1)}.$$ 
The same remark applies to other conditions.
\end{rem}
\begin{rem} We note that such assumptions are natural for distributional coefficients in the sense that regularisations of distributions are moderate. Precisely, by the structure theorems for distributions (see, e.g. \cite{Garet}, \cite{Friedlander}), we know that distributions 
\begin{equation}\label{moder}
  \mathcal{D}'(0,1) \subset \{L^\infty(0,1) -\text{moderate families} \},  
\end{equation}
and we see from \eqref{moder}, that a solution to an initial/boundary problem may not exist in the sense of distributions, while it may exist in the set of $L^\infty$-moderate functions. 
\end{rem}

To give an example, at least for $1\leq p<\infty$, let us take $f\in L^2(0,1)$, $f:(0,1)\to \mathbb{C}$. We introduce the function 
$$\Tilde{f}=\left\{\begin{array}{l}
    f, \text{ on }(0,1),  \\
    0,  \text{ on }\mathbb{R} \setminus (0,1),
\end{array}\right.$$
then $\Tilde{f}:\mathbb{R}\to \mathbb{C}$, and $\Tilde{f}\in \mathcal{E}'(\mathbb{R}).$

Let $\Tilde{f}_\varepsilon=\Tilde{f}*\psi_\varepsilon$ be obtained as the convolution of $\Tilde{f}$ with a Friedrich mollifier $\psi_\varepsilon$, where 
$$\psi_\varepsilon(x)=\frac{1}{\varepsilon}\psi\left(\frac{x}{\varepsilon}\right),\quad \text{for}\,\, \psi\in C^\infty_0(\mathbb{R}),\, \int \psi=1. $$
Then the regularising net $(\Tilde{f}_\varepsilon)$ is $L^p$-moderate for any $p \in [1,\infty)$, and it approximates $f$ on $(0,1)$:
$$0\leftarrow \|\Tilde{f}_\varepsilon-\Tilde{f}\|^p_{L^p(\mathbb{R})}\approx \|\Tilde{f}_\varepsilon-f\|^p_{L^p(0,1)}+\|\Tilde{f}_\varepsilon\|^p_{L^p(\mathbb{R}\setminus (0,1))}.$$
Now, let us introduce the notion of a very weak solution to the initial/boundary problem \eqref{C.p1}-\eqref{C.p3}.

\begin{defn}\label{D2}
Let $p, \, \nu \in \mathcal{D}'(0,1)$.  The net $(u_\varepsilon)_{\varepsilon>0}$ is said to be a very weak solution to the initial/boundary problem  \eqref{C.p1}-\eqref{C.p3} if there exist a $\log$-$L^\infty_1$-moderate regularisation $p_\varepsilon$ of $p$, $L^\infty_1$-moderate regularisation $\nu_\varepsilon$ of $\nu$ with $q_\varepsilon=\nu'_\varepsilon$, $H^2$-moderate regularisation $u_{0,\varepsilon}$ of $u_0,$ and $L^2$-moderate regularisation $u_{1,\varepsilon}$ of $u_1$, such that
\begin{equation}\label{vw1}
         \left\{\begin{array}{l}\partial^2_t u_\varepsilon(t,x)-\partial^2_x u_\varepsilon(t,x)+p_\varepsilon(x)\partial_xu_\varepsilon(t,x)+q_\varepsilon(x) u_\varepsilon(t,x)=0,\,\, (t,x)\in [0,T]\times(0,1),\\
 u_\varepsilon(0,x)=u_{0,\varepsilon}(x),\,\,\, x\in (0,1), \\
\partial_t u_\varepsilon(0,x)=u_{1,\varepsilon}(x), \,\,\, x\in (0,1),\\
u_\varepsilon(t,0)=0=u_\varepsilon(t,1), \quad t\in[0,T],
\end{array}\right.\end{equation}
and $(u_\varepsilon)$, $(\partial_x u_\varepsilon)$ are $L^{2}$-moderate.
\end{defn}

Then we have the following properties of very weak solutions.

\begin{thm}[Existence]\label{Ext}
Let the coefficients $p$, $q$ and initial data $(u_0,\, u_1)$ be distributions in $(0,1)$. Then the initial/boundary problem  \eqref{C.p1}-\eqref{C.p3} has a very weak solution.
\end{thm}
\begin{proof}
Since the formulation of \eqref{C.p1}-\eqref{C.p3} in this case might be impossible in the distributional sense due to issues related to the product of distributions, we replace \eqref{C.p1}-\eqref{C.p3} with a regularised equation. In other words, we regularise $p$, $p'$, $\nu$, $q$, $u_0$, $u_1$, $u'_0$ and $u''_0$ by some corresponding sets $p_\varepsilon$, $p'_\varepsilon$, $\nu_\varepsilon$, $q_\varepsilon$, $u_{0,\varepsilon}$, $u_{1,\varepsilon}$, $u'_{0,\varepsilon}$ and $u''_{0,\varepsilon}$ of smooth functions from $L^ \infty(0,1)$ and $L^2(0,1)$, respectively.

Hence, $p_\varepsilon$ is $\log$-$L^\infty_1$-moderate regularisation of the coefficient $p$, and $\nu_\varepsilon$ with $q_\varepsilon=\nu'_\varepsilon$ is $L^\infty_1$-moderate regularisation of $\nu$, $u_{0,\varepsilon}$ is $H^2$-moderate regularisation of $u_0$ and $u_{1,\varepsilon}$ is $L^2$-moderate regularisation of $u_1$. So by Definition \ref{D1} there exist $N\in \mathbb{N}_0$ and $C_1>0$, $C_2>0$, $C_3>0$, $C_4>0$, $C_5>0$, $C_6>0$, $C_7>0$, $C_8>0$ such that
$$\|p_\varepsilon\|_{L^\infty}\leq C_1|\log{\varepsilon}|^N,\quad \|p'_\varepsilon\|_{L^\infty}\leq C_2\varepsilon^{-N},\quad \|\nu_\varepsilon\|_{L^\infty}\leq C_3\varepsilon^{-N}\quad \|q_\varepsilon\|_{L^\infty}\leq C_4\varepsilon^{-N},$$
$$\|u_{0,\varepsilon}\|_{L^2}\leq C_5\varepsilon^{-N}, \quad \|u_{1,\varepsilon}\|_{L^2}\leq C_6\varepsilon^{-N}, \quad \|u'_{0,\varepsilon}\|_{L^2}\leq C_7\varepsilon^{-N}, \quad \|u''_{0,\varepsilon}\|_{L^2}\leq C_8\varepsilon^{-N}.$$

Now we fix $\varepsilon\in (0,1]$, and consider the regularised problem \eqref{vw1}. Then all discussions and calculations of Theorem \ref{th1} are valid. Thus, by Theorem \ref{th1}, the equation \eqref{vw1} has unique solution $u_\varepsilon(t,x)$ in the space $C^0([0,T]; H^1(0,1))\cap C^1([0,T];L^2(0,1))$.

By Corollary \ref{cor1} there exist $N\in \mathbb{N}_0$ and $C>0$, such that
$$\|u_\varepsilon(t,\cdot)\|_{L^2}\lesssim \exp{\{2\|p_\varepsilon\|_{L^2}\}}\left(\|u_{0,\varepsilon}\|_{L^2}+\|u_{1,\varepsilon}\|_{L^2}\right)\leq C\varepsilon^{-N},$$
\begin{eqnarray*}
\|\partial_x u_\varepsilon(t,\cdot)\|^2_{L^2}
&\lesssim& \exp{\left\{2\|p_\varepsilon\|_{L^1}\right\}}\left\{\left(1+\|\nu_\varepsilon\|^2_{L^2}\left(\|\nu_\varepsilon\|^2_{L^2}+\|p_\varepsilon\|^2_{L^2}+\|p'_\varepsilon\|^2_{L^1}\right)\right)\times\right.\nonumber\\
&\times& \left(\|u''_{0,\varepsilon}\|^2_{L^2}+\|p_\varepsilon\|^2_{L^\infty}\|u'_{0,\varepsilon}\|^2_{L^2}+\left(\|p^2_\varepsilon\|^2_{L^\infty}+\|p'_\varepsilon\|^2_{L^\infty}+\|q_\varepsilon\|^2_{L^\infty}\right)\|u_{0,\varepsilon}\|^2_{L^2}\right.\nonumber\\
&+&\left.\|u_{1,\varepsilon}\|^2_{L^2}\right)+\left.\left(\|p_\varepsilon\|^2_{L^\infty}+\|\nu_\varepsilon\|^2_{L^\infty}\right)\left(\|u_{0,\varepsilon}\|^2_{L^2}+\|u_{1,\varepsilon}\|^2_{L^2}\right)\right\}\leq C\varepsilon^{-N},
\end{eqnarray*}
where the constants in these inequalities are independent of $p$, $p'$, $\nu$, $q$, $u_0$, $u_1$, $u'_0$ and $u''_0$.
Hence, $(u_\varepsilon)$ is $L^2$-moderate, and the proof of Theorem \ref{Ext} is complete.
\end{proof}

\begin{rem}
By
\begin{eqnarray*}
\|\partial_t u_\varepsilon(t,\cdot)\|^2_{L^2}&\lesssim& \exp{\{2\|p_\varepsilon\|_{L^1}\}}\left( \|u''_{0,\varepsilon}\|^2_{L^2}+\|p_\varepsilon\|^2_{L^\infty}\|u'_{0,\varepsilon}\|^2_{L^2}\right.\nonumber\\
&+&\left.\left(\|p^2_\varepsilon\|^2_{L^\infty}+\|p'_\varepsilon\|^2_{L^\infty}+\|q_\varepsilon\|^2_{L^\infty}\right)\|u_{0,\varepsilon}\|^2_{L^2}+\|u_{1,\varepsilon}\|^2_{L^2}\right)\leq C\varepsilon^{-N},
\end{eqnarray*}
we note that the net $\partial_tu_\varepsilon$ is also $L^2$-moderate.
\end{rem}

Describing the uniqueness of the very weak solutions amounts to “measuring” the changes on involved associated nets: negligibility conditions for nets of functions/distributions read as follows:

\begin{defn}[Negligibility]\label{D3}
(i) Let $(u_\varepsilon)$, $(\Tilde{u}_\varepsilon)$ be two nets in $L^2(0,1)$. Then, the net $(u_\varepsilon-\Tilde{u}_\varepsilon)$ is called $L^2$-negligible, if for every $N\in \mathbb{N}$, there exists $C>0$ such that the following condition is satisfied
$$\|u_\varepsilon-\Tilde{u}_\varepsilon\|_{L^2}\leq C \varepsilon^N,$$
for all $\varepsilon\in (0,1]$. In the case where $u_\varepsilon=u_\varepsilon(t,x)$ is a net depending on $t\in [0,T]$, then the negligibility condition can be introduced as 
$$\|u_\varepsilon(t,\cdot)-\Tilde{u}_\varepsilon(t,\cdot)\|_{L^2}\leq C \varepsilon^N,$$
uniformly in $t\in [0,T]$. The constant $C$ can depend on $N$ but not on $\varepsilon$.

(ii) Let $(p_\varepsilon)$, $(\Tilde{p}_\varepsilon)$ be two nets in $L^\infty(0,1)$. Then, the net $(p_\varepsilon-\Tilde{p}_\varepsilon)$ is called $L^\infty$-negligible, if for every $N\in \mathbb{N}$, there exists $C>0$ such that the following condition is satisfied
$$\|p_\varepsilon-\Tilde{p}_\varepsilon\|_{L^2}\leq C \varepsilon^N,$$
for all $\varepsilon\in (0,1]$. 
\end{defn}

Let us state the ``$\varepsilon$-parameterised problems" to be considered:
\begin{equation}\label{un1}
    \left\{\begin{array}{l}\partial^2_t u_\varepsilon(t,x)-\partial^2_x u_\varepsilon(t,x)+p_\varepsilon(x)\partial_xu_\varepsilon(t,x)+ q_\varepsilon(x) u_\varepsilon(t,x)=0,\,\,\, (t,x)\in [0,T]\times(0,1),\\
 u_\varepsilon(0,x)=u_{0,\varepsilon}(x),\,\,\, x\in (0,1), \\
\partial_t u_\varepsilon(0,x)=u_{1,\varepsilon}(x), \,\, x\in (0,1),\\
u_\varepsilon(t,0)=0=u_\varepsilon(t,1),\,\, t\in[0,T],
\end{array}\right.
\end{equation}
and
\begin{equation}\label{un2}
    \left\{\begin{array}{l}\partial^2_t \Tilde{u}_\varepsilon(t,x)-\partial^2_x \Tilde{u}_\varepsilon(t,x)+\Tilde{p}_\varepsilon(x)\partial_x\Tilde{u}_\varepsilon(t,x)+\Tilde{q}_\varepsilon(x) \Tilde{u}_\varepsilon(t,x)=0,\quad (t,x)\in [0,T]\times(0,1),\\
 \Tilde{u}_\varepsilon(0,x)=\Tilde{u}_{0,\varepsilon}(x),\,\,\, x\in (0,1), \\
\partial_t \Tilde{u}_\varepsilon(0,x)=\Tilde{u}_{1,\varepsilon}(x), \,\,\, x\in (0,1),\\
\Tilde{u}_\varepsilon(t,0)=0=\Tilde{u}_\varepsilon(t,1), \quad t\in[0,T].
\end{array}\right.
\end{equation}

\begin{defn}[Uniqueness of the very weak solution]\label{D4}
We say that initial/boundary problem \eqref{C.p1}-\eqref{C.p3} has a unique very weak solution, if
for all $\log$-$L^\infty_1$-moderate nets $p_\varepsilon$, $\Tilde{p}_\varepsilon$, such that $(p_\varepsilon-\Tilde{p}_\varepsilon)$ is $L^\infty$-negligible; $L^\infty_1$-moderate nets $\nu_\varepsilon$, $\Tilde{\nu}_\varepsilon$ with $q_\varepsilon=\nu'_\varepsilon$, $\Tilde{q}_\varepsilon=\Tilde{\nu}'_\varepsilon$ such that $(q_\varepsilon-\Tilde{q}_\varepsilon)$ is $L^\infty$-negligible; for all $H^2$-moderate regularisations $u_{0,\varepsilon},\,\Tilde{u}_{0,\varepsilon}$, such that $(u_{0,\varepsilon}-\Tilde{u}_{0,\varepsilon})$ are $L^2$-negligible and for all $L^2$-moderate regularisations $u_{1,\varepsilon},\,\Tilde{u}_{1,\varepsilon}$, such that $(u_{1,\varepsilon}-\Tilde{u}_{1,\varepsilon})$  are $L^2$-negligible, we have that $u_\varepsilon-\Tilde{u}_\varepsilon$ is $L^2$-negligible.
\end{defn}

\begin{thm}[Uniqueness of the very weak solution]\label{Th-U}
Let the coefficients $p$, $q=\nu'$ and initial data $(u_0,\, u_1)$ be distributions in $(0,1)$. Then the very weak solution to the initial/boundary problem  \eqref{C.p1}-\eqref{C.p3} is unique.
\end{thm}
\begin{proof}
We denote by $u_\varepsilon$ and $\Tilde{u}_\varepsilon$ the families of solutions to the initial/boundary problems \eqref{un1} and \eqref{un2} respectively. Setting $U_\varepsilon$ to be the difference of these nets $U_\varepsilon:=u_\varepsilon(t,\cdot)-\Tilde{u}_\varepsilon(t,\cdot)$, then $U_\varepsilon$ solves
    \begin{equation}\label{unq}
    \left\{\begin{array}{l}\partial^2_t U_\varepsilon(t,x)-\partial^2_x U_\varepsilon(t,x)+p_\varepsilon(x)\partial_xU_\varepsilon(t,x)+q_\varepsilon(x) U_\varepsilon(t,x)=f_\varepsilon(t,x),\\
    \qquad\qquad\qquad\qquad\qquad\qquad\qquad\qquad\qquad\qquad\qquad (t,x)\in [0,T]\times(0,1),\\
 U_\varepsilon(0,x)=(u_{0,\varepsilon}-\Tilde{u}_{0,\varepsilon})(x),\,\,\, x\in (0,1), \\
\partial_t U_\varepsilon(0,x)=(u_{1,\varepsilon}-\Tilde{u}_{1,\varepsilon})(x), \,\,\, x\in (0,1),\\
U_\varepsilon(t,0)=0=U_\varepsilon(t,1),
\end{array}\right.
\end{equation}
where we set $f_\varepsilon(t,x):=(\Tilde{p}_\varepsilon(x)-p_\varepsilon(x))\partial_x\Tilde{u}_\varepsilon(t,x)+(\Tilde{q}_\varepsilon(x)-q_\varepsilon(x))\Tilde{u}_\varepsilon(t,x)$ for the forcing term to the non-homogeneous initial/boundary problem \eqref{unq}.

Passing to the $L^2$-norm of the $U_\varepsilon$, by using \eqref{ec-nh1} we obtain 
\begin{eqnarray*}
\|U_\varepsilon(t,\cdot)\|^2_{L^2}&\lesssim& \exp{\{2\|p_\varepsilon\|_{L^2}\}}\left(\|U_\varepsilon(0,\cdot)\|^2_{L^2}+\|\partial_tU_\varepsilon(0,\cdot)\|^2_{L^2}+2T^2\|f_\varepsilon\|^2_{C([0,T],L^2(0,1))}\right).
\end{eqnarray*}
Since
$$\|f_\varepsilon\|^2_{C([0,T],L^2(0,1))}\leq\|p_\varepsilon-\Tilde{p}_\varepsilon\|^2_{L^\infty}\|\partial_x\Tilde{u}_\varepsilon\|^2_{C([0,T],L^2(0,1))}+ \|q_\varepsilon-\Tilde{q}_\varepsilon\|^2_{L^\infty}\|\Tilde{u}_\varepsilon\|^2_{C([0,T],L^2(0,1))}$$
and using the initial data of \eqref{unq}, we get
\begin{eqnarray*}
\|U_\varepsilon(t,\cdot)\|^2_{L^2}&\lesssim&C\varepsilon^{-N_0}\Big( \|u_{0,\varepsilon}-\Tilde{u}_{0,\varepsilon}\|^2_{L^2}+\|u_{1,\varepsilon}-\Tilde{u}_{1,\varepsilon}\|^2_{L^2}\\
&+&\left.2T^2\|p_\varepsilon-\Tilde{p}_\varepsilon\|^2_{L^\infty}\|\partial_x\Tilde{u}_\varepsilon\|^2_{C([0,T],L^2(0,1))}+2T^2\|q_\varepsilon-\Tilde{q}_\varepsilon\|^2_{L^\infty}\|\Tilde{u}_\varepsilon\|^2_{C([0,T],L^2(0,1))}\right),
\end{eqnarray*}
for some $N_0>0$.
Taking into account the negligibility of the nets $u_{0,\varepsilon}-\Tilde{u}_{0,\varepsilon}$, $u_{1,\varepsilon}-\Tilde{u}_{1,\varepsilon}$, $p_\varepsilon-\Tilde{p}_\varepsilon$ and $q_\varepsilon-\Tilde{q}_\varepsilon$ we get
$$\|U_\varepsilon(t,\cdot)\|^2_{L^2}\leq C_1\varepsilon^{-N_0}\left(C_2\varepsilon^{N_1}+C_3\varepsilon^{N_2}+C_4\varepsilon^{N_3}\varepsilon^{-N_4}+C_5\varepsilon^{N_5}\varepsilon^{-N_6}\right)$$
for some $C_1>0,\,C_2>0,\,C_3>0,\,C_4>0,\,C_5>0,\,N_0,\, N_4,\,N_6\in \mathbb{N}$ and all $N_1,\,N_2,\,N_3,\,N_5\in \mathbb{N}$, since $\Tilde{u}_\varepsilon$ is moderate. Then, for all $M\in \mathbb{N}$ we have
$$\|U_\varepsilon(t,\cdot)\|^2_{L^2}\leq C_M \varepsilon^M.$$
The last estimate holds true uniformly in $t$ , and this completes the proof of Theorem \ref{Th-U}.
\end{proof}

\begin{thm}[Consistency]\label{Th-C} Assume that $p'\in L^2(0,1)$, $q=\nu'$, $\nu \in L^\infty(0,1)$, and let $p_\varepsilon$ be any $\log$-$L^\infty_1$-moderate regularisation of $p$, $\nu_\varepsilon$ be any $L^\infty_1$-moderate regularisation of $\nu$ with $q_\varepsilon=\nu'_\varepsilon$. Let the initial data satisfy $(u_0,\, u_1) \in L^2(0,1)\times L^2(0,1)$. Let $u$ be a very weak solution of the initial/boundary problem \eqref{C.p1}-\eqref{C.p3}. Then for any families $p_\varepsilon$, $q_\varepsilon$, $u_{0,\varepsilon}$, $u_{1,\varepsilon}$ such that $\|u_{0}-u_{0,\varepsilon}\|_{L^2}\to 0$, $\|u_{1}-u_{1,\varepsilon}\|_{L^2}\to 0$, $\|p-p_{\varepsilon}\|_{L^\infty}\to 0$  $\|q-q_{\varepsilon}\|_{L^\infty}\to 0$ as $\varepsilon\to 0$, any representative $(u_\varepsilon)$ of $u$ converges as $$\sup\limits_{0\leq t\leq T}\|u(t,\cdot)-u_\varepsilon(t,\cdot)\|_{L^2(0,1)}\to 0$$ for $\varepsilon\to 0$ to the unique classical solution in $C([0,T];L^2(0,1))$ of the initial/boundary problem \eqref{C.p1}-\eqref{C.p3} given by Theorem \ref{th1}.
\end{thm}

\begin{proof}
For $u$ and for $u_\varepsilon$, as in our assumption, we introduce an auxiliary notation $V_\varepsilon(t, x):= u(t,x)-u_\varepsilon(t,x)$. Then the net $V_\varepsilon$ is a solution to the initial/boundary problem
\begin{equation}\label{51}
    \left\{\begin{array}{l}
        \partial^2_tV_\varepsilon(t,x)-\partial^2_xV_\varepsilon(t,x)+p_\varepsilon(x)\partial_x V_\varepsilon(t,x)+q_\varepsilon(x)V_\varepsilon(t,x)=f_\varepsilon(t,x),\\
        V_\varepsilon(0,x)=(u_0-u_{0,\varepsilon})(x),\quad x\in (0,1),\\
        \partial_tV_\varepsilon(0,x)=(u_1-u_{1,\varepsilon})(x),\quad x\in (0,1),\\
        V_\varepsilon(t,0)=0=V_\varepsilon(t,1), \quad t\in[0,T],
    \end{array}\right.
\end{equation}
where $f_\varepsilon(t,x)=(p_\varepsilon(x)-p(x))\partial_xu(t,x)+(q_\varepsilon(x)-q(x))u(t,x)$. Analogously to Theorem \ref{Th-U} we have that
\begin{eqnarray*}
\|V_\varepsilon(t,\cdot)\|^2_{L^2}&\lesssim& C\|p_\varepsilon\|_{L^\infty}\Big( \|u_{0}-{u}_{0,\varepsilon}\|^2_{L^2}+\|u_{1}-{u}_{1,\varepsilon}\|^2_{L^2}\\
&+&\left.2T^2\|p_\varepsilon-p\|^2_{L^\infty}\|\partial_xu\|^2_{C([0,T],L^2(0,1))}+2T^2\|q_\varepsilon-q\|^2_{L^\infty}\|u\|^2_{C([0,T],L^2(0,1))}\right),
\end{eqnarray*}
Since
$$\|u_{0}-{u}_{0,\varepsilon}\|_{L^2}\to 0,\quad \|u_{1}-{u}_{1,\varepsilon}\|_{L^2}\to 0,\quad \|p_\varepsilon-p\|_{L^\infty}\to 0,\quad \|q_\varepsilon-q\|_{L^\infty}\to 0$$
for $\varepsilon\to 0$ and $u$ is a very weak solution of the initial/boundary problem \eqref{C.p1}-\eqref{C.p3} we get
$$\|V_\varepsilon(t,\cdot)\|_{L^2}\to 0$$
for $\varepsilon\to 0$. This proves Theorem \ref{Th-C}.

\end{proof}

\end{document}